\documentclass[11pt,a4paper]{article}

\usepackage[a4paper,margin=2.2cm]{geometry}

\usepackage[T1]{fontenc}
\usepackage[utf8]{inputenc} 
\usepackage{microtype}

\usepackage{mathtools}
\usepackage{amssymb,amsfonts,amsthm}
\usepackage{bm}
\usepackage{upgreek} 

\usepackage{graphicx}
\usepackage{caption}
\usepackage{booktabs}
\usepackage{siunitx}

\usepackage{enumitem}

\usepackage{tikz}
\usetikzlibrary{
	calc,positioning,patterns,arrows.meta,
	decorations.pathreplacing,decorations.pathmorphing,
	angles,quotes,intersections,backgrounds,shadings
}
\usepackage{tikz-3dplot}
\usepackage{pgfplots}
\usepgfplotslibrary{fillbetween}
\tikzset{snake it/.style={decorate, decoration=snake}}

\usepackage{xcolor}
\usepackage{listings}
\lstdefinestyle{clean}{
	basicstyle=\ttfamily\small,
	frame=single,
	framerule=0.3pt,
	breaklines=true,
	columns=fullflexible,
	keepspaces=true,
	showstringspaces=false
}
\usepackage{titlesec}
\titleformat{\section}{\Large\sffamily\bfseries}{\thesection}{0.8em}{}
\titleformat{\subsection}{\large\sffamily\bfseries}{\thesubsection}{0.8em}{}
\titlespacing*{\section}{0pt}{2.2ex plus .6ex minus .2ex}{1.0ex plus .2ex}
\titlespacing*{\subsection}{0pt}{1.6ex plus .4ex minus .2ex}{0.8ex plus .2ex}

\usepackage{fancyhdr}
\usepackage{hyperref}
\usepackage[nameinlink,noabbrev]{cleveref}
\hypersetup{
	colorlinks=true,
	linkcolor=blue,
	citecolor=blue,
	urlcolor=blue,
	pdftitle={Your Article Title},
	pdfauthor={Your Name}
}
\usepackage{bookmark}

\usepackage[toc,page]{appendix}

\newif\ifdraft
\draftfalse 

\ifdraft
\usepackage[notref]{showkeys}
\usepackage[colorinlistoftodos]{todonotes}
\else
\usepackage[disable]{todonotes}
\fi

\usepackage[normalem]{ulem}

\usepackage{cite}
\numberwithin{equation}{section} 

\newcommand\ds{\displaystyle}

\def\1{{\bf 1}}
\def\wo{\overline}
\def\clb{\color{blue}}
\def\clr{\color{red}}
\def\clm{\color{magenta}}

\def\O{\Omega}

\def\G1{{\bf 1}}

\def\d{\delta}

\def\e{\varepsilon}

\def\C{\mathcal{C}}

\def\sy{\text{sym}}
\def\Te{{\cal T}_\varepsilon}

\def\a{{\alpha}}

\def\Ga{{\bf a}}

\def\ds{\displaystyle}

\def\R{{\mathbb{R}}}
\def\Z{{\mathbb{Z}}}

\def\D{{\mathbb{D}}}
\def\R{{\mathbb{R}}}

\def\Z{{\mathbb{Z}}}

\def\Sc{{\mathcal{S}}}
\def\Pc{{\mathcal{P}}}
\def\Kc{{\mathcal{K}}}

\def\Pc{{\mathcal{P}}}

\def\Jc{{\mathcal{J}}}
\def\Mc{{\mathcal{M}}}

\def\ed{{\varepsilon\delta}}

\def\wh{\widehat }

\def\X{\times }

\def\Ga{{\bf a}}

\def\Gi{{\bf i}}
\def\Gj{{\bf j}}

\def\Gm{{\bf m}}
\def\Gn{{\bf n}}

\def\Gu{{\bf u}}
\def\Gv{{\bf v}}
\def\Gw{{\bf w}}

\def\Gz{{\bf z}}

\def\GA{{\bf A}}

\def\GH{{\bf H}}

\def\GJ{{\bf J}}

\def\GM{{\bf M}}

\def\GR{{\bf R}}

\def\GT{{\bf T}}
\def\GU{{\bf U}}

\def\GV{{\bf V}}

\definecolor{skyblue}{RGB}{135,206,235}
\definecolor{deepskyblue}{RGB}{0, 191, 255}

\usepackage{amsfonts}
\usepackage{graphicx}
\usepackage{epstopdf}
\usepackage{dsfont}
\usepackage{bm}
\usepackage{bbm}
\usepackage{algorithmic}
\usepackage{tikz}
\usepackage{pgfplots}
\usetikzlibrary{patterns}

\usepackage{nicefrac}
\usepackage{mathtools}

\usepackage{ulem}

\usepackage{url, breakurl}
\usepackage{charter} 
\usepackage{longtable}
\usepackage{float}
\usepackage{supertabular}
\usepackage{graphicx}
\usepackage{color}
\usepackage{subfigure}
\usepackage{overpic}
\usepackage{textcomp}
\usepackage[section]{placeins}

\usepackage{enumitem}

\usepackage{siunitx}
\newtheorem{theorem}{Theorem}[section]
\newtheorem{lemma}[theorem]{Lemma}
\newtheorem{corollary}[theorem]{Corollary}
\newtheorem{definition}{Definition}

\newtheorem{remark}{Remark}

\begin{document}
	
	\title{Homogenization of an optimal control problem for nonlocal semilinear elasticity with soft inclusions}
	
	
	%
\author{%
	Amartya Chakrabortty\thanks{Corresponding author. Processes and Materials, Fraunhofer ITWM, Kaiserslautern 67663, Germany. \href{mailto:amartya.chakrabortty@itwm.fraunhofer.de}{amartya.chakrabortty@itwm.fraunhofer.de}}\hskip 10mm
	Abu Sufian\thanks{Department of Mathematics, Universidad de Concepcion, 4070409, Chile. \href{mailto:asufian@udec.cl}{asufian@udec.cl}}
}

	\date{}
	
	\maketitle
	

\begin{abstract}
	This paper investigates the asymptotic analysis of an optimal control problem (OCP) posed on a high-contrast elastic medium with soft periodic inclusions, governed by a semilinear elasticity system with a nonlocal term. The domain consists of a connected matrix phase and a soft inclusion phase. The model depends on two independent small parameters: the periodicity $\varepsilon>0$ and the contrast $\delta>0$, and the distributed control acts only in the inclusion region. We consider an $L^2$-tracking cost on the displacement and analyze the limit as $(\varepsilon,\delta)\to(0,0)$ in the regime
	\[ 
	 \lim_{(\varepsilon,\delta)\to(0,0)}\frac{\delta}{\varepsilon}=\kappa\in(0,+\infty].
	\]
	First, we derive the homogenized (limit) state system associated with this scaling. We then formulate the limit OCP and prove that the limit of the microscopic optimal controls is an optimal control for the limit problem, using a $\Gamma$-convergence approach.
\end{abstract}

\section*{Introduction}
Smart origami-inspired metamaterials \cite{gries} and advanced textiles \cite{orlik01} offer a platform for designing structures with
programmable deformation patterns. A common engineering mechanism is to combine a connected stiff skeleton
(e.g.\ panels or fibers) with comparatively soft regions (e.g.\ hinges, joints, or inclusions), so that the
overall response is governed by a pronounced contrast of stiffness across the microstructure. In many
applications, one seeks to actively tune the deformation by embedding actuation or pre-strain in the
soft phase, while the stiff phase primarily carries and transmits load. This naturally leads to models in
which the control/actuation is supported only in the soft inclusions.

A widely used modeling viewpoint for such materials is to encode pre-strain (or growth/swelling) through a
multiplicative decomposition of the deformation gradient into an elastic part and a non-elastic part, see \cite{prestrain02}.
Under small-deformation and small-prestrain assumptions, the resulting linearized equations for the
displacement contain an additive term generated by the pre-strain (see \cite{Amartya}), which can be recast as an effective
source term in the equilibrium equations (see \cite{GOW2}). In this work, we adopt this perspective and treat the resulting
effective source term as a distributed control acting only in the soft inclusions. This provides a
mathematically tractable framework for shape/response optimization while remaining consistent with common
linearization principles used in the mechanics of pre-strained or swelling media.

Beyond the contrast and microstructure, many soft active materials exhibit an additional feature: their
mechanical response may depend not only on the local strain, but also on a global measure of the
deformation amplitude. Such global coupling can arise, for instance, from an underlying interaction with a
surrounding medium (e.g.\ solvent redistribution in gels, network effects, or mean-field-type feedback),
or from modeling choices that penalize large overall deformation to reflect saturation/stiffening effects.
To capture this in a minimal way while keeping the elastic operator local, we include a nonlocal semilinear term of the form $\alpha\|u\|_{L^2(\Omega)}^p u$ with $p\geq 2$. Equivalently, this term is generated by adding
to the total energy a convex contribution proportional to $\|u\|_{L^2(\Omega)}^{p+2}$, so that the resulting
Euler--Lagrange equation couples the local elastic equilibrium with a global amplitude-dependent response.

Homogenization for partial differential equations with high-contrast coefficients has attracted sustained attention over the past decades, in particular for diffusion and conductivity-type models. Early contributions include \cite{MR620783,MR926151}, where strongly contrasting coefficients were analyzed in connection with effective properties of composite media. Since then, a broad literature has developed on high-contrast homogenization and the emergence of nonstandard (often nonlocal) effective models.

Beyond diffusion, high-contrast effects also arise in elasticity and more general variational settings. In \cite{MR1952924}, the macroscopic behavior of an elastic composite with a strong component embedded in a weak matrix was studied; the homogenized limit exhibits nonlocal features that depend sensitively on the geometry of the stiff phase. Related high-contrast phenomena for conductivity were investigated in \cite{MR3000657}, which considers a medium consisting of highly conductive vertical fibers embedded in a poorly conductive matrix. Further examples include \cite{MR3342695}, addressing an elliptic variational problem on a pillar-type oscillating domain with an insulating layer surrounding a highly conductive core, and \cite{MR3450748}, where a hyperbolic PDE with contrasting diffusive coefficients leads to a coupled two-scale limit system incorporating both macro- and micro-scale effects. For additional references on high-contrast homogenization we refer to \cite{gaudiello2021limit,briane2007homogenization,MR3251918} and the references therein.

The homogenization of optimal control problems (OCPs) governed by PDE constraints has also been extensively studied. For elliptic OCPs with oscillating coefficients (possibly appearing both in the state equation and in the cost functional), see for instance \cite{MR1466916,MR451153}. The homogenization of OCPs in porous media was investigated in \cite{MR1666365}. For OCPs with high-contrast diffusion coefficients, we refer to \cite{AbuCon,MR4552077}, where the limiting problems may exhibit nonlocal effects. In these works, the principal part of the state equation is given by the Laplacian; see also \cite{MR451153,durante2010asymptotic,bidhan01} and the references therein.

Homogenization of an approximate controllability problem (ACP) constrained by semilinear state systems was initiated in \cite{Conca01} on fixed domains, and later extended to perforated domains in the cases where the holes do not intersect the control region \cite{Conca02} and where the holes intersect the control region \cite{Conca03}. In these works, the semilinearity enters through a local term involving a $\C^1$-function $S_1:\R\to\R$ satisfying
\begin{equation}\label{N1}
S_1(0)=0,
\qquad
\exists\,\gamma>0 \ \text{such that}\ 
0\le \frac{S_1(s)}{s}\le \gamma,
\quad \forall\, s\in\R\setminus\{0\}.
\end{equation}
Homogenization of OCPs in domains with oscillating boundaries was studied in \cite{Nanda01,Nanda02,Nanda03,Nanda04}, where the control acts on the fixed (non-oscillating) part of the boundary/domain. In these works, the semilinearity is also introduced via a local term involving a $\C^2$-function $S_2:\R\to\R$ such that
\begin{equation}\label{N2}
S_2(0)=0,
\qquad
0<C_1\le S_2'(s)\le C_2,
\quad \text{and}\quad S_2'' \ \text{is bounded}.
\end{equation}
In the analysis of above OCPs, the indirect method, which concerns the
asymptotic behavior of the corresponding optimality systems is used. 
In \cite{NL01}, homogenization study of a semilinear OCP is presented the direct method, relying on $\Gamma$-convergence arguments. Their semilinearity is assumed to satisfy a global growth/Lipschitz estimate of at most quadratic type,
together with monotonicity closedness. Moreover, since the leading operator is uniformly coercive/elliptic (i.e., no high-contrast degeneracy), the states admit uniform $H^1$-bounds, hence by the compact embedding strong $L^2$ convergence. This strong $L^2$-compactness is the key ingredient to pass to the limit in such semilinear terms.

This paper studies the asymptotic behavior of an optimal control problem (OCP) posed on a high-contrast domain in $\R^N$, governed by a semilinear elasticity system with a nonlocal (at least cubic) term, using the indirect method and a form of $\Gamma$-convergence. The high-contrast domain is a bounded domain with Lipschitz boundary $\O\subset \R^N$. It is decomposed into two parts: a connected stiff part $\O^2_\e$ with relative stiffness of order $1$, and a soft part $\O^1_\e$ with relative stiffness of order $\d^2$. The microscopic control $\theta_\e\in L^2(\O^1_\e)^N$ acts only on the soft part $\O^1_\e$ (extended by zero to $\O$).

The governing system (understood in the weak sense; see Section~\ref{S02} for details) reads, for $p\ge 2$ and $\alpha>0$,
\begin{equation}\label{Main02}
	\left\{\begin{aligned}
		-\nabla \cdot \sigma_\ed(u_\ed)+\alpha\|u_\ed\|^p_{ L^2(\O)}\, u_\ed&=f(x)+\1_{\O^1_\e}\theta_\e,\quad &&\text{in }\O,\\
		u_\ed&=0,\quad &&\text{on }\Gamma_0,\\
		\sigma_\ed(u_\ed)\Gn&=0,\quad&&\text{on }\Gamma_1,
	\end{aligned}\right.
\end{equation}
where $f=(f_i)$ denotes the body force.

The main aim of this paper is to analyze the asymptotic behavior, as $(\e,\d)\to(0,0)$, of an OCP that seeks to match a prescribed displacement profile using external forces supported in the soft inclusions. Let $u_d\in L^2(\O)^N$ be a given desired displacement. For each admissible control $\theta_\e$, we denote by $u_\ed(\theta_\e)\in\GU$ the unique weak solution of \eqref{Main02} (equivalently, of \eqref{MainW01}). We consider the reduced cost functional
\begin{equation*}
	\Gj_{\ed}(\theta_\e)
	=\frac12\int_\O |u_\ed(\theta_\e)-u_d|^2\,dx
	+\frac{\gamma}{2}\int_{\O^1_\e}|\theta_\e|^2\,dx,
\end{equation*}
where $\gamma>0$ is a regularization parameter. The OCP consists in finding an optimal control $\Theta_\e\in L^2(\O^1_\e)^N$ such that
\begin{equation*}
	\Gj_{\ed}(\Theta_\e)=\inf_{\theta_\e\in L^2(\O^1_\e)^N}\Gj_{\ed}(\theta_\e).
\end{equation*}

The results of the paper can be divided into two parts. First, we derive the limit (homogenized) state system associated with \eqref{Main02} in the asymptotic regime
\begin{equation}\label{R1}
	(\e,\d)\to(0,0),\quad \text{and}\quad \kappa=\lim_{(\e,\d)\to(0,0)}\frac{\d}{\e}\in(0,+\infty].
\end{equation}
Using periodic unfolding, suitable energy arguments and a form of $\Gamma$-convergence, we identify the limit displacement, the associated correctors, the limit total energy and characterize that the limit energy depends on the parameter $\kappa$. Moreover, in the case $\kappa=+\infty$ we obtain a fully scale-separated homogenized system. The main tool for asymptotic analysis is the periodic unfolding operator, which was introduced in \cite{CDG2} and further developed in \cite{CDG1}. For an introduction to homogenization and elasticity, see \cite{Olenik,Donato,twoscale1,twoscale2,CDG}.
 
Second, we analyze the asymptotic behavior of the microscopic OCP as $(\e,\d)\to(0,0)$ with $\d/\e\to\kappa$. We establish existence of microscopic optimal controls and prove convergence of the adjoint state and optimal controls. Moreover, we derive the limit OCP and show that any limit point of a sequence of microscopic optimal controls is an optimal control for the limit problem. The proof relies on a $\Gamma$-convergence approach for the reduced cost functionals, together with stability and strong-convergence properties for the unfolded state variables.

Finally, we emphasize a further novelty of the present work concerns the structure of the nonlocal semilinearity.
In several related contributions on homogenization of OCPs with semilinear state equations, the nonlinearity
enters through a scalar functional of the state combined with a scalar nonlinearity (see, for example,
the functions $S_1$ and $S_2$ in \eqref{N1}--\eqref{N2}), under global growth/monotonicity assumptions. In contrast,
our model involves the Hilbert-space operator
\[
u \mapsto \alpha\|u\|_{L^2(\Omega)}^{p}\,u,\quad p\geq 2
\]
which couples all spatial points through the global $L^2$-amplitude of the displacement.
As a consequence, passing to the limit in this term cannot be achieved from weak convergence alone; it requires
strong $L^2$ convergence (equivalently, convergence of $L^2$ norms), which we obtain through energy convergence
arguments in the unfolding framework. Moreover, the Fr\'echet derivative of this operator contains a
nonlocal rank-one contribution,
\[
D(\|u\|^p u)[h]=\|u\|^p h + p\|u\|^{p-2}(u,h)_{L^2(\Omega)}\,u,
\]
leading to additional global coupling in the linearized and adjoint systems.
These features, combined with the high-contrast elastic operator, make the analysis more delicate than in the
scalar diffusion setting. In particular, uniform $H^1$ bounds typically fail because the elastic energy degenerates
in the soft phase (due to the $\delta^2$ scaling). To overcome this difficulty, we introduce a suitable decomposition
(see Lemmas~\ref{L2}--\ref{L3}) and establish convergence of the total energy via a $\Gamma$-convergence argument.

The paper is organized as follows. General notation and the description of the domain are introduced in
Sections~\ref{N-1} and \ref{S02}, respectively. In Section~\ref{S03}, we prove existence, uniqueness, and
a priori estimates for the weak solution (equivalently, the minimizer) of the $\varepsilon$--$\delta$-dependent
semilinear elasticity system. In Section~\ref{S04}, we establish existence of an optimal control for the
$\varepsilon$--$\delta$-dependent OCP and show that the control-to-state mapping is globally Lipschitz. The main homogenization results for the state equation are presented in
Section~\ref{prototype}.  Finally, in
Section~\ref{S06}, we analyze the asymptotic behavior of the $\varepsilon$--$\delta$ OCP.

\section{Notation}\label{N-1}
Below, we give some general notations which will be used throughout the paper.
\begin{itemize}
	\item Let $N\geq 2$ and $Y=(0,1)^N$. Let $Y_1\subset Y$ bounded open connected set (domain) with Lipschitz boundary such that $\wo{Y_1}\subset Y$ and $Y_2= Y\setminus\overline{Y_1}$. Moreover, we assume that $Y_2$ is a bounded domain with Lipschitz boundary. We denote the interface by $\Lambda=\wo{Y_1}\cap \wo{Y_2} =\partial Y_1$. 
	\item  $\R^{N\X N}$ is the space of $N\times N$  real matrices, $\R_\sy^{N\X N}$ is the space of real $N\X N$ symmetric matrices.
	\item The strain tensor is given by
	$$e(u)={1\over 2}\big((\nabla u)^T+\nabla u\big),\quad \text{for all $u\in H^1(\O)^N$}.$$
	\item The boundary $\partial\O$ of $\O$ is partitioned into two disjoint parts $\Gamma_0$ and $\Gamma_1$ with the surface measure of $\Gamma_0$ being strictly positive. We set $\Gn(\cdot)$ the unit outward normal on $\partial\O$.
	\item $\alpha,\gamma > 0$ and $(i,j,k,l)\in \{1,2,\ldots, N\}^4$ (if not specified).
	\item We have used the convention ${1\over +\infty}=0$.
	\item By convention in all the estimates we simply write $L^2(\O)$ instead of $L^2(\O)^N$ or $L^2(\O)^{N\times N}$, we write the complete spaces when we give weak or strong convergence.
	\item In this paper, we use the Einstein convention of summation over repeated indices and we denote by $C$ a generic constant independent of $\e$ and $\d$.
\end{itemize}
\section{Domain description and initial setup}\label{S02}
Let $\O\subset\R^N$ be a bounded domain with Lipschitz boundary.
We define the sets $\O_\e^1$ and $\O^2_\e$ as follows:
$$\O_{\e}^1=\bigcup_{k\in \Kc_\e} (\e k+ \e Y_1),\quad \O_\e^2=\O\setminus \overline{\O^1_\e},\qquad  \Kc_\e =\big\{k\in \Z^N\; |\; \e(k+ Y)\subset \O\big\}.$$  Note that the holes $k\e+\e Y_1$, $k\in \Kc_\e$, do not intersect the boundary of $\O$.
We denote the interface by $\Lambda_\e=\wo{\O^1_\e}\cap \wo{\O^2_\e} =\partial \O^1_\e $.\\
Observe that $\O^2_\e$ is bounded domain with Lipschitz boundary and $\partial \O\subset \partial \O^2_\e$. We denote the indicator function by
$$\1_{\O^1_\e}(x)=\left\{\begin{aligned}
1,\quad &x\in \O^1_\e,\\
0,\quad &x\in\GR^N\setminus \O^1_\e.
\end{aligned}\right.\quad 
 \1_{\O\X Y_1}(x,y) =\left\{\begin{aligned}
1,\quad &(x,y)\in \O\X Y_1,\\
0,\quad &(x,y)\in\O\X(\GR^N\setminus Y_1).
\end{aligned}\right.$$
We define the two 4th order Hooke's tensors $A\in [L^\infty(Y)]^{N\X N\X N\X N}$ satisfying the following
\begin{itemize}
	\item Symmetric: $A_{ijkl}= A_{jikl}= A_{klij}$.
	\item Coercivity and  condition ($K>0$)
	\begin{equation}\label{Coe01}
		\begin{aligned}
			A_{ijkl}(y) S_{ij} S_{kl}\geq K| S|^2_F,\quad \text{for a.e. $y\in Y$ and for all $S\in \R^{N\X N}_\sy$}.
		\end{aligned}	
	\end{equation}
	Here $|\cdot|_F$ denotes the standard Frobenius norm. 
	\item We set
	$$A_{\e,ijkl}(x)= A_{ijkl}\left({x\over \e}\right),\quad  \text{for a.e. $x\in\O$}.$$
	\item Let $\GA_\ed=\d^2 A_\e \1_{\O^1_\e}+ A_\e(1-\1_{\O^1_\e})$.
	\item Let us define the stress tensor $\sigma_\ed$ by
	$$\sigma_{\ed,ij}(u)=\GA_{\ed,ijkl}e_{kl}(u),\quad \text{for all $u\in H^1(\O)^N$}.$$
\end{itemize}
\section{Existence of unique solution and a priori estimates}\label{S03}
The problem is to determine the displacement field $u_\ed=(u_{\ed,i})$ satisfying the system \eqref{Main02}. By virtue of the assumptions on $\GA_\ed$, we define the bilinear form $\Ga_\ed$ by
\begin{equation*}
\Ga_\ed(u,w)=\int_{\O}\sigma_{\ed,ij}(u)e_{ij}(w)\,dx,\quad \forall\,u,w\in H^1(\O)^N.
\end{equation*}
We also define the global nonlinear term as
$$g(u,w)=\alpha\|u\|^p_{L^2(\O)}\int_\O u\cdot w\,dx,\quad \forall\,u,w\in H^1(\O)^N.$$
So, the weak form of the problem \eqref{Main02} is given by: For a given function $f\in L^2(\O)^N$ and $\theta_\e\in L^2(\O_\e^1)^N$, find the solution ${u}_\ed\in \GU$ such that
\begin{equation}\label{MainW01}
\Ga_\ed(u_\ed,w)+g(u_\ed,w)=\ell_\e(w)=\int_\O f\cdot w\,dx+\int_{\O^1_\e}\theta_\e\cdot w\,dx,\quad \forall\,w\in \GU,
\end{equation}
where the set of admissible displacement is given by
$$\GU=\left\{u\in H^1(\O)^N\,|\; u=0\;\text{on $\Gamma_0$}\right\}.$$
The associated minimization problem is 
\begin{equation}\label{MainE01}
	\left\{\begin{aligned}
		&\text{Find $u_\ed\in\GU$ such that}\\
		&\hskip 3cm\Gm_\ed=\GT_\ed(u_\ed)=\inf_{w\in\GU}\GT_\ed(w),
	\end{aligned}\right.
\end{equation}
where the total energy to this nonlinear elasticity problem is given by
$$\GT_\ed(w)={1\over 2}\Ga_\ed(w,w)+{1\over p+2}g(w,w)-\ell_\e(w),\quad\text{for all $w\in\GU$}.$$
Since $\GT_{\e\d}$ is Fréchet differentiable and strictly convex on $\GU$ (for $\alpha>0$), its unique minimizer in $\GU$ coincides with the unique weak solution of \eqref{MainW01}.
\begin{theorem}\label{Th1}
	For every $\e,\d>0$, there exist a unique minimizer $u_\ed\in \GU$ to the problem \eqref{MainE01}, i.e.
	$$\GT_\ed(u_\ed)=\Gm_\ed=\inf_{w\in\GU}\GT_\ed(w).$$
\end{theorem}
\begin{proof}
	Observe that
	$$\Gm_\ed\leq \GT_\ed(0)=0,$$
	so, we have $\Gm_\ed\in [-\infty,0]$. 
	
	{\bf Step 1.} We prove $\Gm_\ed\in(-\infty,0]$.
	
	First using the assumption \eqref{Coe01} and Korn's inequality, we have for any $w\in\GU$
	\begin{equation}\label{l2}
		\Ga_\ed(w,w)\geq C(\d)\|w\|^2_{H^1(\O)}.
	\end{equation}
	The positive constant  $C(\d)$  depends on $\d$. 
		
	Using Cauchy-Schwarz and Poincare inequality, we have for $w\in \GU$
	\begin{equation}\label{l1}
		|\ell_\e(w)|\leq C(\|f\|_{L^2(\O)}+\|\theta_\e\|_{L^2(\O^1_\e)})\|w\|_{L^2(\O)}\leq C(\|f\|_{L^2(\O)}+\|\theta_\e\|_{L^2(\O^1_\e)})\|w\|_{H^1(\O)}.
	\end{equation}
	The positive constant $C$ is independent of $\e$ and $\d$.
	
	We also have
	$$g(w,w)=\alpha\|w\|^{p+2}_{L^2(\O)}.$$
	Observe that for all $w\in\GU$ satisfying
	\begin{equation}\label{l3}
		\GT_\ed(w)\leq \GT_\ed(0)=0,
	\end{equation}
	we have using \eqref{l2}--\eqref{l1}
	\begin{multline}\label{AP01}
		{C(\d)\over 2}\|w\|^2_{H^1(\O)}+{\a\over p+2}\|w\|^{p+2}_{L^2(\O)}\leq {1\over 2}\Ga_\ed(w,w)+{1\over p+2}g(w,w)\\
		\leq |\ell_\e(w)|\leq C(\|f\|_{L^2(\O)}+\|\theta_\e\|_{L^2(\O^1_\e)})\|w\|_{H^1(\O)},
	\end{multline}
	which imply
	$$\|w\|_{H^1(\O)}\leq C(\d)(\|f\|_{L^2(\O)}+\|\theta_\e\|_{L^2(\O^1_ \e)}),\quad |\ell_\e(w)|\leq C(\d).$$
	Then, the above inequalities give for all $w\in\GU$ satisfying \eqref{l3}
	\begin{equation}\label{l4}
		\|w\|_{H^1(\O)}\leq C(\d)(\|f\|_{L^2(\O)}+\|\theta_\e\|_{L^2(\O^1_\e)}),\quad	\text{and}\quad
		-k(\e,\d)\leq -\ell_\e(w)\leq \GT_\ed(w)\leq 0.
	\end{equation}
	The positive constant $C(\d)$ depend on $\d$ and $k(\e,\d)$ depends on $\d$ and $\theta_\e$.

	The above inequalities give $\Gm_\ed\in(-\infty,0]$.
	
	{\bf Step 2.} We prove the existence of minimizer.
	
	Note the following: The quadratic elastic part $w\mapsto\Ga_\ed(w,w)$ is convex and continuous on $\GU$, hence weakly lower semicontinuous. The nonlinear term $w\mapsto g(w,w)$ is convex functional on a Hilbert space, so weakly lower semicontinuous. The right hand side term $-\ell_\e(w)$ is weakly continuous. Thus, $\GT_\ed$ is weakly lower semicontinuous on $\GU$. 
	
	Using the estimate \eqref{l4}, there exist a minimizing sequence $\{w_n\}_n\subset \GU$ satisfying \eqref{l3} and
	$$\Gm_\ed=\inf_{w\in\GU}\GT_\ed(w)=\liminf_{n\to+\infty}\GT_\ed(w_n).$$
	We also have the following convergence at-least for a subsequence (denoted by the same subscript)
	$$w_n\rightharpoonup u_\ed,\quad\text{weakly in $H^1(\O)^N$}.$$
	Since, $\GT_\ed$ is weak lower semicontinuous, we obtain
	$$\GT_\ed(u_\ed)\leq \liminf_{n\to+\infty}\GT_\ed(w_n)=\Gm_\ed.$$
	So, we have existence of  minimizer in $\GU$ and the fact that $\Gm_\ed$ is a minimum by sequencial criteria.
	
	{\bf Step 3.} We prove uniquesness of the minimizer.
	
	Let us consider the operator 
	$$\langle F(u),w\rangle:=\Ga_\ed(u,w)+g(u,w),\quad\forall\,u,w\in\GU.$$
	Then, $u\in\GU\mapsto F(u)\in\GU^\ast$, where $\GU^\ast$ is the dual space of $\GU$.
	
	Observe that for $u,w\in\GU$, we have (using \eqref{l2})
	\begin{align*}
		&\langle F(u)-F(w),u-w\rangle\\
		&\hskip1cm=\Ga_\ed(u-w,u-w)+g(u,u-w)-g(w,u-w)\\
		&\hskip 1cm \geq C\|u-w\|^2_{H^1(\O)}.
	\end{align*}
	since
	\begin{align*}
		0 &\le g(u,u-w)-g(w,u-w)
		= \alpha\big(\|u\|_{L^2(\Omega)}^{p}u-\|w\|_{L^2(\Omega)}^{p}w,\;u-w\big)_{L^2(\Omega)}\\
		&=\frac{\alpha}{2}\Big(\|u\|_{L^2(\Omega)}^{p}-\|w\|_{L^2(\Omega)}^{p}\Big)
		\Big(\|u\|_{L^2(\Omega)}^{2}-\|w\|_{L^2(\Omega)}^{2}\Big)
		+\frac{\alpha}{2}\Big(\|u\|_{L^2(\Omega)}^{p}+\|w\|_{L^2(\Omega)}^{p}\Big)
		\|u-w\|_{L^2(\Omega)}^{2}.
	\end{align*}
	
	Let $u_\ed,v_\ed$ be two minimizer of \eqref{MainE01} in $\GU$, then, we have from the above expression, we have
	\begin{equation*}
		0=\langle F(u_\ed)-F(v_\ed),u_\ed-v_\ed\rangle\geq C\|u_\ed-v_\ed\|^2_{H^1(\O)},
	\end{equation*}
	since $F(u_\ed)=F(v_\ed)=\ell_\e$ in $\GU^\ast$ using the variational form \eqref{MainW01}. So, we get $u_\ed=v_\ed$ a.e. in $\O$. Hence, there exist a unique minimizer (which also the weak solution to the problem \eqref{MainW01}) to the problem \eqref{MainE01}.
	
	This completes the proof.
\end{proof}	
The unique solution corresponding to the given $\theta_\e$ will be denoted by $u_\ed(\theta_\e)$.
\begin{lemma}[Preliminary estimates]\label{L2}
	Let $u\in\GU$ be a displacement, then there exist displacements $\Gu\in\GU$ and $\Gw$ in $H^1_0(\O^1_\e)^N$, extended by $0$ in $\O^2_\e$,  such that
	\begin{equation}\label{ExRF}
		u=\Gu+\Gw\quad\text{a.e. in $\O$}
	\end{equation}
	and 
	\begin{equation}\label{46}
		\begin{aligned}
			\|e(\Gu)\|_{L^2(\O)}&\leq C\|e(u)\|_{L^2(\O_\e^2)},\\
			\|e(\Gw)\|_{ L^2(\O)}&\leq C\|e(u)\|_{L^2(\O^1_\e)}+C\|e(u)\|_{L^2(\O_\e^2)},\\
			\|\Gw\|_{L^2(\O^1_\e)}+\e\|\nabla \Gw\|_{L^2(\O^1_\e)} &\leq C\e\|e( \Gw)\|_{L^2(\O)}.
		\end{aligned}
	\end{equation}
	Moreover, we have the following Korn-type inequality
	\begin{equation}\label{NK01}
		\|u\|_{L^2(\O)}\leq C\left(\e\|e(u)\|_{L^2(\O^1_\e)}+\|e(u)\|_{L^2(\O^2_\e)}\right).
	\end{equation}
	The constant(s) do not depend on $\e$ and $\d$.
\end{lemma}
\begin{proof}
	We recall a standard extension result from \cite{Olenik,CDG}, since $\O^2_\e\subset \O$ is bounded connected open set with Lipschitz boundary: There exists an extension operator ${\bm\Pc}_\e$ from $H^1(\O^2_\e)^N$ into $H^1(\O)^N$ satisfying
	\begin{equation}
		\begin{aligned}
			\forall w\in H^1(\O^2_\e)^N,\qquad &{\bm\Pc}_\e(w)\in H^1(\O)^N,\qquad {\bm \Pc}_\e(w)_{|\O^2_\e}=w,\\
			&\big\|e\big({\bm\Pc}_\e(w)\big)\big\|_{L^2(\O)}\leq C\|e(w)\|_{L^2(\O^2_\e)}.
		\end{aligned}
	\end{equation}
	The constant does not depend on $\e$ and $\d$.
	
	Let $u\in \GU$. Then, we denote by $\Gu={\bm\Pc}_\e\left({u}_{|\O^2_\e}\right)$ and $\Gw=u-\Gu$.
	
	The estimate \eqref{46}$_1$ is due to the previous extension result \eqref{ExRF}. Since $u=\Gu+\Gw$ and using \eqref{46}$_1$, we have \eqref{46}$_2$.\\
	We prove \eqref{46}$_2$. We define $\Gw_k(y)=\Gw(\e k+ \e y)$ for all $y\in Y_1$. Since $\Gw$ satisfy \eqref{46}$_2$, we have using Korn's inequality in the fixed domain $Y_1$
	$$\|\Gw_k\|_{L^2(Y_1)}+\|\nabla_y \Gw_k\|_{L^2(Y_1)}\leq C\|e_y(\Gw_k)\|_{L^2(Y_1)}.$$
	The constant does not depend on $k\in \Kc_\e$, it only depends on $Y_1$.
	Since
	$$\nabla_y\Gw_k(y)=\e\nabla\Gw(\e k+ \e y),\quad\text{for a.e. $y\in Y_1$},$$
	we obtain using change of variable
	$$\|\Gw\|^2_{L^2(\e k+\e Y_1)}+ \e^2\|\nabla \Gw\|^2_{L^2(\e k+\e Y_1)}\leq C \e^2\|e(\Gw)\|^2_{L^2(\e k+\e Y_1)}.$$
	Adding for all $k\in\Kc_\e$ give \eqref{46}$_3$. 
	
	Since $ u=\Gu+\Gw$ and using the estimates \eqref{46}, we obtain
	$$\|u\|_{L^2(\O)}\leq C\|e(\Gu)\|_{L^2(\O)}+C\e\|e(\Gw)\|_{L^2(\O)}\leq C\left(\e\|e(u)\|_{L^2(\O^1_\e)}+\|e(u)\|_{L^2(\O^2_\e)}\right).$$
	This completes the proof.
\end{proof}
As a consequence of the above lemma, we have the following Korn-type inequalities:
		\begin{lemma}\label{L3}
			The unique solution $u_\ed\in\GU$ corresponding to $\theta_\e\in L^2(\O_\e^1)^N$ of \eqref{MainW01} and let $\Gu_\ed(\theta_\e)$, $\Gw_\ed(\theta_\e)$ be the fields from Lemma \ref{L2}. Then, we have the following
			\begin{equation}\label{E01}
				\begin{aligned}
					\d\|e(u_\ed(\theta_\e))\|_{L^2(\O^1_\e)}+\|e(u_\ed(\theta_\e))\|_{L^2(\O^2_\e)}&\leq C\left(\|f\|_{L^2(\O)}+\|\theta_\e\|_{L^2(\O^1_\e)}\right),\\
					\|\Gu_\ed(\theta_\e)\|_{H^1(\O)}+\|e(\Gu_\ed(\theta_\e))\|_{L^2(\O)}&\leq C\left(\|f\|_{L^2(\O)}+\|\theta_\e\|_{L^2(\O^1_\e)}\right),\\
					\|\Gw_\ed(\theta_\e)\|_{L^2(\O)}+\e\|\nabla \Gw_\ed(\theta_\e)\|_{L^2(\O)}+\e\|e(\Gw_\ed(\theta_\e))\|_{L^2(\O)}&\leq {C\e\over \d}\left(\|f\|_{L^2(\O)}+\|\theta_\e\|_{L^2(\O^1_\e)}\right).
				\end{aligned}
			\end{equation}
			The constant(s) do not depend on $\e$ and $\d$.
		\end{lemma}
		\begin{proof}
			Let $\theta_\e\in L^2(\O^1_\e)^N$ be the control. For simplicity, we drop the $(\theta_\e)$ in each displacements.
			
		In the main problem \eqref{MainW01}, taking $w= u_\ed$ and using the coercivity of $\GA_\ed$ in \eqref{Coe01} along with assumption on forces and \eqref{AP01} gives
		$$K\left(\d^2\|e(u_\ed)\|^2_{L^2(\O^1_\e)}+\|e(u_\ed)\|^2_{L^2(\O^2_\e)}\right)+\a\|u_\ed\|^{p+2}_{L^2(\O)}\leq C\left(\|f\|_{L^2(\O)}+\|\theta_\e\|_{L^2(\O^1_\e)}\right)\|u_\ed\|_{L^2(\O)}.$$
		The Korn-type inequality \eqref{NK01}
		imply
		\begin{multline*}
			\left(\d^2\|e(u_\ed)\|^2_{L^2(\O^1_\e)}+\|e(u_\ed)\|^2_{L^2(\O^2_\e)}\right)\\
			\leq C\left(\|f\|_{L^2(\O)}+\|\theta_\e\|_{L^2(\O^1_\e)}\right)\left(\e\|e(u_\ed)\|_{L^2(\O^1_\e)}+\|e(u_\ed)\|_{L^2(\O^2_\e)}\right).
		\end{multline*}
		Due to \eqref{R1}, there exist a $C>0$ independent of $\e$ and $\d$ such that
		\begin{equation}\label{R11}
			\e\leq C\d.
		\end{equation}
	 So, we obtain
		\begin{multline*}
			\left(\d\|e(u_\ed)\|_{L^2(\O^1_\e)}+\|e(u_\ed)\|_{L^2(\O^2_\e)}\right)^2
			\leq 2\left(\d^2\|e(u_\ed)\|^2_{L^2(\O^1_\e)}+\|e(u_\ed)\|^2_{L^2(\O^2_\e)}\right)\\
			\leq 2C\left(\|f\|_{L^2(\O)}+\|\theta_\e\|_{L^2(\O^1_\e)}\right)\left(\d\|e(u_\ed)\|_{L^2(\O^1_\e)}+\|e(u_\ed)\|_{L^2(\O^2_\e)}\right).
		\end{multline*}
		The above inequality gives \eqref{E01}$_1$. Then, using the estimates \eqref{46}$_{3,4}$, we get
		$$\begin{aligned}
			 \d \|e(\Gw_\ed)\|_{L^2(\O^1_\e)} +\|e(\Gu_\ed)\|_{L^2(\O)}\leq C\left(\|f\|_{L^2(\O)}+\|\theta_\e\|_{L^2(\O^1_\e)}\right).
		\end{aligned}$$
		Now, since $\Gu_\ed$ and $\Gw_\ed$ satisfy the boundary conditions \eqref{46}$_{1,2}$, using Korn's inequality and the estimate \eqref{46}$_5$, we obtain
		$$\|\Gu_\ed\|_{H^1(\O)}+{\d \over \e}\|\Gw_\ed\|_{L^2(\O)}+\d\|\nabla \Gw_\ed\|_{L^2(\O)}\leq C\left(\|f\|_{L^2(\O)}+\|\theta_\e\|_{L^2(\O^1_\e)}\right).$$
		This completes the proof.
		\end{proof}
\section{$\e,\d$-dependent optimal control problem}\label{S04}
In this section, we define the optimal control problem. Let use consider the cost functional 
$$\GJ_\e(u,\theta_\e)={1\over 2}\|u-u_d\|^2_{L^2(\O)}+{\gamma\over 2}\|\theta_\e\|^2_{L^2(\O^1_\e)},\quad \forall\,(u,\theta_\e)\in \GU\X \GV,$$
where $\GV=L^2(\O^1_\e)^N$. Let us define the reduced cost function as
$$\Gj_\ed(\theta_\e)=\GJ_\e(u_\ed(\theta_\e),\theta_\e),\quad \forall\,\theta_\e\in\GV.$$
The optimal control problem (OCP) reads as
\begin{equation}\label{OCPM01}
	\begin{aligned}
		&\Gi_\ed=\inf_{\theta_\e\in\GV}\big\{\Gj_\ed(\theta_\e)=\GJ_\e(u_\ed(\theta_\e),\theta_\e)\,|\,\text{subjected to $u_\ed(\theta_\e)$ satisfying \eqref{MainW01}}\big\},
	\end{aligned}
\end{equation}
where $u_\ed(\theta_\e)$ is the unique weak solution to the problem \eqref{MainW01}, corresponding to $\theta_\e$.
\begin{theorem}\label{Th2}
	For each $\e,\d>0$, the OCP \eqref{OCPM01} admits a solution.
\end{theorem}
\begin{proof}
As
		$$\Gi_\ed=\inf_{\theta_\e\in\GV}\GJ_\e(u_\ed(\theta_\e),\theta_\e)=\inf_{\theta_\e\in\GV}\Gj_\ed(\theta_\e),$$
		we have $\Gi_\ed\in [0,+\infty)$ since $\Gi_\ed\leq \Gj_\ed(0)<+\infty$. So, there exist a minimizing sequence $\{\theta_{\e,n}\}_{n}\subset\GV$ such that
		\begin{equation}\label{42}
				\Gi_\ed=\Gj_\ed(\theta_{\e,n})\leq \Gj_\ed(0),\quad \liminf_{n\to+\infty}\Gj_\ed(\theta_{\e,n})=\Gi_\ed.
			\end{equation}
		Since, using \eqref{NK01} and \eqref{E01}, we have
		$$\Gj_\ed(0)={1\over 2}\|u_\ed(0)-u_d\|^2_{L^2(\O)}\leq C,$$
		which imply using \eqref{l4}$_1$
		\begin{equation*}
				\|\theta_{\e,n}\|_{L^2(\O^1_\e)}\leq C ,\quad \|u_{\ed,n}(\theta_{\e,n})\|_{H^1(\O)}\leq C,
			\end{equation*}
		where $u_{\ed,n}(\theta_{\e,n})$ is the unique solution of \eqref{MainW01} corresponding to $\theta_{\e,n}$.
		The constant is independent of $n$. 
		
		So, we obtain there exist $\Theta_\e\in\GV$ and $u_\ed\in\GU$ such that
		\begin{equation*}
				\begin{aligned}
						\theta_{\e,n} &\rightharpoonup \Theta_\e,\quad &&\text{weakly in $L^2(\O^1_\e)^N$},\\
						u_{\ed,n} &\rightharpoonup u_\ed,\quad &&\text{weakly in $H^1(\O)^N$ and strongly in $L^2(\O)^N$}.
					\end{aligned} 
			\end{equation*}
		From the above convergences, we get for $w\in\GU$
		\begin{equation*}
				\lim_{n\to+\infty}\int_{\O^1_\e}\theta_{\e,n}\cdot w\,dx=\int_{\O^1_\e} \Theta_\e\cdot w,\quad \lim_{n\to+\infty}g(u_{\ed,n},w)=g(u_\ed,w),\quad \lim_{n\to+\infty}\Ga_\ed(u_{\ed,n},w)=a_\e(u_\e,w),
			\end{equation*}
		so, we get $u_\ed\in\GU$ is the unique solution corresponding to $\Theta_\e$, i.e. $u_\ed=u_\ed(\Theta_\e)$. Using the weakly lower semi continuity of $L^2$ and $H^1$-norm with \eqref{42}, we get 
		$$\Gj_\ed(\Theta_\e)\leq\liminf_{n\to+\infty}\Gj_\ed(\theta_{\e,n})=\Gi_\e.$$
		So, $\Theta_\e$ is an optimal control. This completes the proof.
\end{proof}
For $\a>0$, the uniqueness of the optimal control is not guaranteed in general, since the reduced problem is non-convex as the control to state map is nonlinear.
 Let $\Theta_\e\in \GV$ be an  optimal control with $\wo u_\ed=u_\ed(\Theta_\e)\in\GU$ being the unique weak solution of \eqref{MainW01} corresponding to $\Theta_\e$.

Observe that $\Gj_{\ed}(\Theta_\e)\leq \Gj_{\ed}(0)$. Then, from the fact that $u_\ed(0)=\Gu_\ed(0)+\Gw_\ed(0)$ along with the estimates \eqref{E01}, we have
\begin{equation}\label{UBOC}
	\begin{aligned}
		\|\wo u_\ed-u_d\|^2_{L^2(\O)}+\gamma\|\Theta_\e\|^2_{L^2(\O^1_\e)}
		&\leq \|u_\ed(0)-u_d\|^2_{L^2(\O)}\\
		&\leq \|\Gu_\ed(0)\|^2_{L^2(\O)}+\|\Gw_\ed(0)\|^2_{L^2(\O)}+\|u_d\|^2_{L^2(\O)}\\
		&\leq {C\e^2\over \d^2}.
	\end{aligned}
\end{equation}
The constant is independent of $\e$ and $\d$.\\
Since $\kappa\in(0,+\infty]$, we have that there exists a $C>0$ independent of $\e$ and $\d$ such that (see also \eqref{R11})
\begin{equation}\label{UBOCC}
		\|\wo u_\ed-u_d\|_{L^2(\O)}+	\|\Theta_\e\|_{L^2(\O^1_\e)}\leq {C_1\e\over \d}\leq C.
\end{equation}		
We call $(\wo u_\ed,\Theta_\e)$ an optimal pair. 			Let us define the control to state map $\Sc_\ed:\GV\to\GU$ as
$$\Sc_\ed(\theta)=u_\ed(\theta),\quad \forall\,\theta\in \GV,$$
where $u_\ed(\theta)$ is the unique solution to \eqref{MainW01} corresponding to $\theta$.
\begin{theorem}\label{Th03}
	The mapping $\Sc_\ed$ is Lipschitz continuous from $\GV$ to $\GU$, i.e.~there exists a constant $L_\ed>0$ (dependent on $\e$ and $\d$) such that
	\begin{equation}\label{CS01}
		\|u^1_{\ed}-u^2_\ed\|_{H^1(\O)}\leq L_\ed\|\theta_\e^1-\theta_\e^2\|_{L^2(\O^1_\e)}.
	\end{equation}
	Moreover, there exist $L>0$ independent of $\e$ and $\d$ such that
	\begin{equation}\label{CS02}
		\|u^1_{\ed}-u^2_\ed\|_{L^2(\O)}\leq L\|\theta_\e^1-\theta_\e^2\|_{L^2(\O^1_\e)},
	\end{equation}
	where $\theta^i\in \GV$ and $\Sc_\ed(\theta^i)=u^i_\ed$ for $i=1,2$.
\end{theorem}
\begin{proof}
	Let $\theta_\e^1,\theta_\e^2\in \GV$ and denote by
	$u_\ed^i:=\Sc_\ed(\theta_\e^i)\in\GU$ $(i=1,2)$ the corresponding unique solutions of \eqref{MainW01}, i.e.
	\[
	\Ga_\ed(u_\ed^i,w)+\a\|u_\ed^i\|_{L^2(\O)}^{p}(u_\ed^i,w)
	=(f,w)+(\theta_\e^i,w)_{L^2(\O_\e^1)}
	\qquad\forall w\in\GU.
	\]
	Subtracting the two identities yields, for all $w\in\GU$,
	\begin{equation}\label{eq:diff_state}
		\Ga_\ed(u^1_\ed-u^2_\ed,w)
		+\a\big(\|u_\ed^1\|_{L^2(\O)}^{p}u_\ed^1-\|u_\ed^2\|_{L^2(\O)}^{p}u_\ed^2,\;w\big)
		=(\theta^1_\e-\theta^2_\e,w)_{L^2(\O_\e^1)}.
	\end{equation}
%
	The nonlinear operator $G(u)=\a\|u\|_{L^2(\O)}^{p}u$ for all $u\in\GV$ is monotone and satisfy
	\begin{equation}\label{eq:monotone_B}
		\big(G(u)-G(v),\,u-v\big)\ge 0
		\qquad\forall u,v\in \GV.
	\end{equation}
	Taking $w=u^1_\ed-u^2_\ed$ in \eqref{eq:diff_state}. Using \eqref{eq:monotone_B} we obtain
	\[
	\Ga_\ed(u^1_\ed-u^2_\ed,u^1_\ed-u^2_\ed)
	+\big(G(u_\ed^1)-G(u_\ed^2),\,u^1_\ed-u^2_\ed\big)
	=(\theta^1_\e-\theta^2_\e,u^1_\ed-u^2_\ed)_{L^2(\O_\e^1)},\]
	which imply
	\[\Ga_\ed(u^1_\ed-u^2_\ed,u^1_\ed-u^2_\ed)\le (\theta^1_\e-\theta^2_\e,u^1_\ed-u^2_\ed)_{L^2(\O_\e^1)}.\]
	Using the coercivity of $\Ga_\ed$ together with the inequality \eqref{NK01}, there exist a constant $C>0$ independent of $\e$ and $\d$, such that
	\begin{multline*}
		C\|u^1_\ed-u^2_\ed\|^2_{L^2(\O)}\leq \Ga_\ed(u^1_\ed-u^2_\ed,u^1_\ed-u^2_\ed)\le (\theta^1_\e-\theta^2_\e,u^1_\ed-u^2_\ed)_{L^2(\O_\e^1)}\\
		\leq C\|\theta^1_\e-\theta^2_\e\|_{L^2(\O^1_\e)}\|u^1_\ed-u^2_\ed\|_{L^2(\O)},
	\end{multline*}
	which give the estimate \eqref{CS02} with constant independent of $\e$ and $\d$. Similarly, using the coercivity of $\Ga_\ed$ together with Korn's inequality and continuous embedding (restriction estimate) $H^1(\O)^N\hookrightarrow L^2(\O)^N$, we get
	$$C(\e,\d)\|u^1_\ed-u^2_\ed\|^2_{H^1(\O)}\leq C\|\theta^1_\e-\theta^2_\e\|_{L^2(\O^1_\e)}\|u^1_\ed-u^2_\ed\|_{H^1(\O)},$$
	which imply \eqref{CS01} with constant dependent on $\e$ and $\d$.
%
	Therefore the control-to-state mapping $\Sc_\ed:\GV\to\GU$ is globally Lipschitz.
\end{proof}

\section{Homogenization of the high-contrast state system \eqref{MainW01}}\label{prototype}
\subsection{Unfolding operator}
The convergences are done via the periodic unfolding operator $\Te$ for homogenization in $\O$. 
Below we recall the definition of the periodic unfolding operator for functions defined in $\O$, respectively. For the complete properties of the operator, see \cite{CDG}, specifically see Proposition 1.12, Theorem 1.36, Corollary 1.37 and Proposition 1.39 in \cite{CDG}. \begin{definition}
	For every measurable function $\phi$ on $\O$ the unfolding operator $\Te$ is defined by 
	$$ \Te(\phi) (x,y) \doteq \left\{ \begin{aligned}
		&\phi \left(\e\left[{x\over \e}\right] + \e y\right) \qquad &&\hbox{for a.e. }\; (x,y)\in  \wh\O_\e\times Y,\\
		&0 \quad &&\hbox{for a.e. }\; (x,y)\in  (\O\setminus\wh\O_\e)\times Y,
	\end{aligned}\right.$$
\end{definition} 
where 
$$\wh \O_\e=\hbox{interior}\left\{\bigcup_{k\in\Kc_\e}(\e k+\e \wo Y)\right\}.$$
The unfolding operator is a continuous linear operator from $L^2(\O)$ into $L^2(\O\X Y)$ which satisfies
$$\|\Te(\phi)\|_{L^2(\O\X Y)}\leq C\|\phi\|_{L^2(\O)}\quad\text{for every}\quad \phi\in L^2(\O).$$
The constants $C$ do not depend on $\e$.
Moreover, for every $\psi\in H^1(\O)$ one has (see \cite[Proposition 1.35]{CDG})
$$ \nabla_y\Te(\psi)(x,y)=\e\Te(\nabla \psi)(x,y) \quad \hbox{a.e. in}\quad \O\X Y.$$
Since $\O$ is a bounded domain with Lipschitz boundary, we have
$$\left|\int_{\O}\psi\,dx-{1\over |Y|}\int_{\O\X Y}\Te(\psi)\,dxdy\right|\to 0,\quad\text{as $\e\to0$},$$
so, we have
\begin{equation}\label{uci}
	\int_{\O}\psi\,dx\simeq{1\over |Y|}\int_{\O\X Y}\Te(\psi)\,dxdy.
\end{equation}
\subsection{Two-scale limit system}
We define the following space
$$ \GH^1_{per}(Y)\doteq\left\{\phi\in H^1_{per}(Y)\,|\, \phi=0\text{ on $Y_2$}\right\}.$$
Let $f\in L^2(\O)^N$ and $(\wo u_\ed,\Theta_\e)$ be an optimal pair. Then, due to the inequality \eqref{UBOCC}, there exist a $C\in\R^+$ independent of $\e$ and $\d$ such that
\begin{equation}\label{OCC01}
	\|\Theta_\e\|_{L^2(\O^1_\e)}\leq C,\quad \text{and}\quad \quad \Te(\Theta_\e) \rightharpoonup \wh \Theta \text{ weakly in }L^2(\Omega\times Y_1)^N,  
\end{equation}
where $\Theta_\e\in L^2(\O^1_\e)^N$ and $\wh\Theta\in L^2(\Omega\times Y_1)^N$.  
Then, for the corresponding fields $(\wo\Gu_\ed,\wo\Gw_\ed)$, we have the following estimates using \eqref{E01}
\begin{equation}\label{EM01}
	\begin{aligned}
		\|\wo\Gu_\ed\|_{H^1(\O)}&\leq C,\qquad {\d\over \e}\|\wo\Gw_\ed\|_{L^2(\O^1_\e)}+\d\|\nabla \wo\Gw_\ed\|_{L^2(\O^1_\e)}\leq C.
	\end{aligned}
\end{equation}
The constant is independent of $\e$ and $\d$.

While analyzing the optimal control problem in Section \ref{S06}, observe that $ \Theta_\e\in L^2(\O_\e^1)^N$ is the optimal microscopic control.

As $(\e,\d)\to (0,0)$, with $\kappa\in(0,+\infty]$, we get\footnote{In all the lemmas and theorems below, we extract  a subsequence  of $\{\ed\}_{\e,\d}$ (still denoted by $\{\ed\}_{\e,\d}$)  in order to get the desired convergences.}
\begin{lemma}\label{L04}
	There exist $\Gu_0\in H^1(\O)^N$ and $\wh U\in L^2(\O; H^1_{per,0}(Y))^N$ with
	$$\Gu_0=0,\quad \text{on $\Gamma_0$},$$
	such that
	\begin{equation}\label{Con01}
		\begin{aligned}
			\wo\Gu_\ed &\rightharpoonup \Gu_0,\quad&&\text{weakly in $H^1(\O)^N$},\\
			\Te(\wo\Gu_\ed)&\to \Gu_0,\quad &&\text{strongly in $L^2(\O\X Y)^N$},\\
			\Te(\nabla \wo\Gu_\ed) &\rightharpoonup \nabla \Gu_0+\nabla_y\wh U,\quad &&\text{weakly in $L^2(\O\X Y)^{N\X N}$}.
		\end{aligned}
	\end{equation}
	Moreover,
	there exist $\wh W\in L^2(\O; \GH^1_{per}(Y))^N$ such that
	\begin{equation}\label{Con02}
		\begin{aligned}
			{\d\over \e}\Te(\wo\Gw_\ed) &\rightharpoonup \wh W,\quad&&\text{weakly in $L^2(\O\X Y)^N$},\\
			\d\Te(\nabla \wo\Gw_\ed) &\rightharpoonup \nabla_y\wh W,\quad&&\text{weakly in $L^2(\O\X Y)^N$}.
		\end{aligned}
	\end{equation}
	Furthermore, we have
	\begin{equation}\label{Con03}
		\begin{aligned}
			\Te(\wo u_\ed) &\rightharpoonup \Gu_0+{1\over \kappa}\wh W,\quad&&\text{weakly in $L^2(\O\X Y)^N$}.
		\end{aligned}
	\end{equation}
\end{lemma}
\begin{proof}
	The convergences \eqref{Con01} are a direct consequence of the estimates \eqref{EM01}$_1$ along with the property of unfolding operator (see Corollary 1.37 in \cite{CDG}). The convergences \eqref{Con02} are a consequence of the estimate \eqref{EM01}$_2$ and of the properties of the unfolding operator given in Proposition 1.12, and Theorem 1.36 of \cite{CDG}. The convergences \eqref{Con03} are a consequence of \eqref{Con01}--\eqref{Con02} along with the fact that $\ds{\e\over \d} \to \ds{1\over \kappa}$.    This completes the proof.
\end{proof}
As a consequence of the above lemma, we get the limit two-scale problems. Before we give the limit problems, we define the following limit spaces
\begin{equation*}
	\begin{aligned}
		&\D=\GU\X L^2(\O;\GH^1_{per}(Y))^N\X L^2(\O; H^1_{per,0}(Y))^N,\quad \D_0=\GU \X L^2(\O;\GH^1_{per}(Y))^N.
	\end{aligned}
\end{equation*}
The two-scale limit energy is given by: for all $(w_0,W_1,W_2)\in\D$ and $\kappa\in(0,+\infty]$
\begin{equation}\label{EXE01}
	\GT_\kappa(w_0,W_1,W_2)={1\over 2}\Ga_1(W_1)+{1\over 2}\Ga_2(w_0,W_2)+{1\over p+2}g_\kappa(w_2,W_1)-\ell_\kappa(w_0,W_1)
\end{equation}
with
$$\begin{aligned}
	\Ga_1(W_1)&=\int_{\O\X Y}A(y)e_y(W_1):e_y(W_1)\,dydx,\quad 	g_\kappa(w_0,W_1)=\alpha\|w_0+{1\over \kappa}W_1\|^{p+2}_{L^2(\O\X Y)},\\
	\Ga_2(w_0,W_2)&=\int_{\O\X Y_2} A(y)\big(e(w_0)+e_y(W_2)\big):\big(e(w_0)+e_y(W_2)\big)\,dy dx,\\
	\ell_\kappa(w_0,W_1)&=\int_{\O\X Y}(f+\1_{\O\X Y_1}\wh\Theta)\cdot \left(w_0+{1\over \kappa} W_1\right)\,dydx.
\end{aligned}$$
\begin{theorem}\label{Th3}
	There exist a unique minimizer to the problem
	$$\Gm_\kappa=\min_{(w_0,W_1,W_2)\in\D}\GT_\kappa(w_0,W_1,W_2).$$
\end{theorem}
\begin{proof}
	First, we have
	$$\Gm_\kappa\leq \GT_\kappa(0,0,0)=0 \implies \Gm_\kappa\in[-\infty,0].$$
	Let us consider $(w_0,W_1,W_2)\in \D$ such that $\GT_\kappa(w_0,W_1,W_2)\leq 0$. Then, using the coercivity of $A$ with the inequality \eqref{Ex01} and Korn's inequality, we get
	$$ \|w_0\|^2_{H^1(\O)}+\|W_2\|^2_{L^2(\O;H^1(Y))}\leq C\Ga_2(w_0,W_2),\quad \|W_1\|^2_{L^2(\O;H^1(Y))}\leq C\Ga_1(W_1).$$
	We also have
	$$| \ell_\kappa(w_0,W_1)|\leq C(\|f\|_{L^2(\O)}+\|\wh\Theta\|_{L^2(\O\X Y_1)})\left(\sqrt{\Ga_2(w_0,W_2)}+\sqrt{\Ga_1(W_1)}\right),$$
	which together with the fact that $\GT_\kappa(w_0,W_1,W_2)\leq 0$ and $g_\kappa(w_0,W_1)\geq0$, give
	$$\Ga_1(W_1)+\Ga_2(w_0,W_2)\leq \sqrt{2}C(\|f\|_{L^2(\O)}+\|\wh\Theta\|_{L^2(\O\X Y_1)})\sqrt{\Ga_2(w_0,W_2)+\Ga_1(W_1)}.$$
	So, we obtain
	\begin{equation*}\label{LES01}
		\begin{aligned}
			\|w_0\|^2_{H^1(\O)}+\|W_2\|^2_{L^2(\O;H^1(Y))}+ \|W_1\|^2_{L^2(\O;H^1(Y))}&\leq  C\Ga_2(w_0,W_2)+C\Ga_1(W_1)\\
			&\leq C(\|f\|_{L^2(\O)}+\|\wh\Theta\|_{L^2(\O\X Y_1)})^2.
		\end{aligned}
	\end{equation*}
	Then, we have $\Gm_\kappa\in(-\infty,0]$. Then, proceeding as in Theorem \ref{Th1}, we get the existence of unique minimizer in $\D$. This completes the proof.
\end{proof}
Finally, we have the convergence of the total energy.
\begin{theorem}\label{ENH01}
	We have 
	\begin{equation}\label{ECM01}
		\lim_{(\e,\d)\to(0,0)}\Gm_\ed=\Gm_\kappa=\GT_\kappa(\Gu_0,\wh W,\wh U).
	\end{equation}
\end{theorem}
\begin{proof}
	The proof is presented in $3$ steps using a form of $\Gamma$-convergence.
	
	{\bf Step 1.} We show that
	\begin{equation}
		\Gm_\kappa\leq \liminf_{(\e,\d)\to(0,0)}\Gm_\ed.
	\end{equation}
	Let $\{\wo u_\ed\}_{\e,\d}$ be the sequence of unique solutions of the elasticity problem \eqref{MainW01} in $\GU$. Then, we extend $(\wo u_\ed)_{|\O^2_\e}$ to $\wo\Gu_\ed\in H^1(\O)^N$ using the extension result \eqref{ExRF} and we have the estimates \eqref{46} and \eqref{E01} for $\wo\Gu_\ed$ and $\wo\Gw_\ed$. Then, with the assumption on the forces, we have the estimates \eqref{EM01}. Then, we get the following convergences at least for a subsequence (as in Lemma \ref{L04})
	\begin{equation}\label{II02}
		\begin{aligned}
			\Te\big(e(\wo{u}_\ed)\big) &\rightharpoonup e(\Gu_0)+e_y(\wh U),\quad&&\text{weakly in $L^2(\O\X Y_2)^{N\X N}$},\\
			\d\Te\big(e(\wo u_\ed)\big)& \rightharpoonup e_y(\wh W),\quad&&\text{weakly in $L^2(\O\X Y_1)^{N\X N}$},\\
			\Te\big(\wo u_\ed\big)& \rightharpoonup \Gu_0+{1\over \kappa}\wh W,\quad&&\text{weakly in $L^2(\O\X Y)^{N}$},
		\end{aligned}
	\end{equation}
	with $(\Gu_0,\wh W,\wh U)\in \D$.
	
	So, first we transorm the energy $\GT_\ed$ using the unfolding operator (see the property \eqref{uci}), then the above convergences combined with \eqref{OCC01} and the weak lower semicontinuity of the functional $\GT_\ed$ give
	$$\Gm_\kappa\leq \GT_\kappa(\Gu_0,\wh W,\wh U)\leq \liminf_{(\e,\d)\to(0,0)}\GT_\ed(\wo u_\ed)=\liminf_{(\e,\d)\to(0,0)}\Gm_\ed.$$
	This completes Step 1.
	
	{\bf Step 2.} 	We prove
	\begin{equation}\label{He01}
		\begin{aligned}
			\limsup_{(\e,\d)\to(0,0)}\Gm_\ed\leq \GT_\kappa(w_0,W_1,W_2),\quad&&&\forall\, (w_0,W_1,W_2)\in\D.
		\end{aligned}	
	\end{equation}
	For that, we construct $\{w_\ed\}_{\e,\d}\subset \GU$ from $(w_0,W_1,W_2)\in\D$ such that
	\begin{equation}\label{59}
		\GT_\ed(\wo u_\ed)\leq \GT_\ed(w_\ed),\quad \lim_{(\e,\d)\to(0,0)}\GT_\ed(w_\ed)=\GT_\kappa(w_0,W_1,W_2).
	\end{equation}
	Observe that by density argument, it is enough to prove the above statement for $(w_0,W_1,W_2)\in\D\cap \C^1(\wo\O)^N\X \C^1(\wo{\O}\X \wo{Y})^N\X\C^1(\wo{\O}\X \wo{Y})^N$. 
	
	Let $(w_0,W_1,W_2)\in \D \cap \left(\C^1(\O)^N\X \C^1(\wo{\O}\X \wo{Y})^N\X\C^1(\wo{\O}\X \wo{Y})^N\right)$.  Then we set
	\begin{equation}\label{Test01}
		w_\ed(x)=\left\{\begin{aligned}
			&w_0(x)+\e W_2\left(x,\left\{{x\over \e}\right\}\right),\quad&&\text{for $x\in \O^2_\e$},\\
			&w_0(x)+\e W_2\left(x,\left\{{x\over \e}\right\}\right)+{\e\over \d}W_1\left(x,\left\{{x\over \e}\right\}\right),\quad&&\text{for $x\in\O^1_\e$},
		\end{aligned}\right.
	\end{equation} 
	which give
	$$\nabla w_\ed=\left\{\begin{aligned}
		&\nabla w_0+\e \nabla_x W_2+\nabla_y W_1,\quad&&\text{in $\O^2_\e$},\\
		&\nabla w_0+\e \nabla_x W_2+\nabla_y W_2+{\e\over \d}\nabla_x W_1+{1\over \d}\nabla_y W_1,\quad&&\text{in $\O^1_\e$},
	\end{aligned}\right.$$ 
	which imply
	$$\begin{aligned}
		\Te\big(e(w_\ed)\big)&\to e(w_0)+e_y(W_2),\quad &&\text{strongly in $L^2(\O\X Y_2)^{N\X N}$},\\
		\Te\big(\d e(w_\ed)\big) &\to e_y(W_1),\quad&&\text{strongly in $L^2(\O\X Y_1)^{N\X N}$},\\
		\Te(w_\ed) &\to w_0+{1\over \kappa}W_1,\quad &&\text{strongly in $L^2(\O\X Y)^N$}.
	\end{aligned}$$
	So, with the above above convergences and proceeding as in Step 1, we obtain \eqref{59}. Since $(w_0,W_1,W_2)\in\D$ is arbitrary, we obtain \eqref{He01}.

Finally, combining Step 1 and 2, we get
$$\Gm_\kappa\leq  \GT_\kappa(\Gu_0,\wh W,\wh U)\leq \liminf_{(\e,\d)\to(0,0)}\GT_\ed(\wo u_\ed)\leq \limsup_{(\e,\d)\to(0,0)}\GT_\ed(\wo u_\ed)\leq \Gm_\kappa,$$
by taking the unique minimizer of $\Gm_\kappa$ in place of $(w_0,W_1,W_2)\in\D$. Hence, we get the energy convergence \eqref{ECM01} and the fact that $(\Gu_0,\wh W,\wh U)\in\D$ is the unique minimizer of $\GT_\kappa$ in $\D$. Due to the existence of unique minimizer, we have that the convergences \eqref{II02} holds for the whole sequence.
This completes the proof.
\end{proof}
As a consequence of the above energy convergence, 
\begin{corollary}\label{Cor1}
	We have the following strong convergence for the whole sequence
	\begin{equation}\label{SCM}
		\Te\big(\wo u_\ed\big) \to \Gu_0+{1\over \kappa}\wh W,\quad\text{strongly in $L^2(\O\X Y)^{N}$}.
	\end{equation}
\end{corollary}
\begin{proof}
	
	From the convergence of total energy \eqref{ECM01} and the weak convergence \eqref{II02}, we have
	$$\lim_{(\e,\d)\to(0,0)}\left[{1\over 2}\Ga_\ed(\wo u_\ed,\wo u_\ed)+{1\over p+2}g(\wo{u}_\ed,\wo u_\ed)\right]={1\over 2}\left[\Ga_1(\wh W)+\Ga_2(\Gu_0,\wh U)\right]+{1\over p+2}g_\kappa(\Gu_0,\wh W),$$
	which combined with the weak lower-semicontinuity, non-negativity of $\Ga_\ed(\cdot,\cdot)$, $\|\cdot\|^{p+2}_{L^2(\O)}$ and the convergences \eqref{II02} give
	$$\liminf_{(\e,\d)\to(0,0)}\Ga_\ed(\wo u_\ed,\wo u_\ed)=\Ga_1(\wh W)+\Ga_2(\Gu_0,\wh U),\quad \liminf_{(\e,\d)\to(0,0)}g(\wo{u}_\ed,\wo u_\ed)=g_\kappa(\Gu_0,\wh W).$$
	Then, we get
	\begin{multline*}
		\limsup_{(\e,\d)\to(0,0)}{1\over p+2}g(\wo u_\ed,\wo u_\ed)=\limsup_{(\e,\d)\to(0,0)}\left[{1\over 2}\Ga_\ed(\wo u_\ed,\wo u_\ed)+{1\over p+2}g(\wo{u}_\ed,\wo u_\ed)-{1\over 2}\Ga_\ed(\wo u_\ed,\wo u_\ed)\right]\\
		=\limsup_{(\e,\d)\to(0,0)}\left[{1\over 2}\Ga_\ed(\wo u_\ed,\wo u_\ed)+{1\over p+2}g(\wo{u}_\ed,\wo u_\ed)\right]-{1\over 2}\liminf_{(\e,\d)\to(0,0)}\Ga_\ed(\wo u_\ed,\wo u_\ed)\\
		={1\over p+2}g_\kappa(\Gu_0,\wh W),
	\end{multline*}
	which imply
	$$\lim_{(\e,\d)\to(0,0)}\a\|\Te(\wo u_\ed)\|^{p+2}_{L^2(\O\X Y)}=\a\|\Gu_0+{1\over\kappa}\wh W\|^{p+2}_{L^2(\O\X Y)}.$$
Moreover, by \eqref{II02} we have
$T_\e(\bar u_{\e\delta}) \rightharpoonup \Gu_0+\frac1\kappa \wh W$ weakly in $L^2(\Omega\times Y)^N$.
Since $t\mapsto t^{p+2}$ is strictly increasing on $[0,\infty)$, the previous limit implies
$\|T_\e(\bar u_{\e\delta})\|_{L^2(\Omega\times Y)}\to \|\Gu_0+\frac1\kappa \wh W\|_{L^2(\Omega\times Y)}$.
Therefore, by uniform convexity of $L^2(\Omega\times Y)^N$, we conclude the strong convergence \eqref{SCM}.

	This completes the proof.
\end{proof}

The unique minimzer $(\Gu_0,\wh W,\wh U)\in\D$ satisfy the following two-scale system (in variational form)
\begin{multline}\label{L01}
	\int_{\O\X Y_2} A(y)\big(e(\Gu_0)+e_y(\wh U)\big):\big(e(w_0)+e_y(W_2)\big)\,dx dy\\
	+\int_{\O\X Y}A(y) e_y(\wh W):e_y(W_1)\,dxdy+\alpha\|\Gu_0+{1\over \kappa}\wh W\|^p_{L^2(\O\X Y)}\int_{\O\X Y}(\Gu_0+{1\over \kappa}\wh W)\cdot (w_0+{1\over \kappa}W_1)\,dydx\\
	=\int_{\O\X Y}\big(f+\1_{\O\X Y_1}\wh\Theta\big)\cdot \left(w_0+{1\over \kappa} W_1\right)\,dydx,
\end{multline}
for all $(w_0,W_1,W_2)\in\D$.

\subsection{Homogenization and cell problems}\label{SS44}
In this subsection, we derive the homogenized limit problems for $\kappa\in(0,+\infty]$.

In \eqref{L01}, we replace $w_0=0$ and $W_1=0$ with localizing, we obtain\footnote{Observe that $\wh U\in H^1_{per,0}(Y)^N$, since $\nabla_y\wh U$ enters the variational form, we can choose $\wh U\in H^1_{per,0}(Y_2)^N$.}
\begin{equation}\label{514}
	\begin{aligned}
		&\text{Find $\wh U\in H^1_{per,0}(Y_2)^N$ such that}\\
		&\hskip 10mm \int_{Y_2}A(y)(e(\Gu_0)+e_y(\wh U)):e_y(W_2)\,dy=0,\quad\text{for all $W_2\in H^1_{per,0}(Y_2)^N$}.
	\end{aligned}
\end{equation}
This shows that $\wh U$ can be expressed in terms of the elements of the tensor $e(\Gu_0)$ and some correctors. First, we define the following $N\X N$ symmetric matrices by
$$\GM^{ij}_{kl}={1\over 2}\left(\d_{ki}\d_{lj}+\d_{kj}\d_{li}\right),\quad i,j,k,l\in\{1,\dots, N\},$$
where $\d_{ij}$ is the Kronecker symbol. The cell problems are given by
\begin{equation*}
	\begin{aligned}
		&\text{Find $\chi^2_{ij}\in H^1_{per,0}(Y_2)^N$ such that}\\
		&\hskip 10mm \int_{Y_2} A(y)(\GM^{ij}+e_y(\chi^2_{ij})):e_y(W_2)\,dy=0,\quad\text{for all $W_2\in H^1_{per,0}(Y_2)^N$}.
	\end{aligned}
\end{equation*}
The above equation implies
\begin{equation}\label{Warping}
	\wh U(x,y)=\sum_{i,j=1}^N e_{ij}(\Gu_0)(x)\chi^2_{ij}(y),\quad\text{for a.e. $(x,y)\in \O\X Y_2$}.
\end{equation}
Observe that using the expression \eqref{Warping}, we have
\begin{equation}\label{Hom01}
	A^{hom}_{ijkl}=\int_{Y_2}\left(A_{ijkl}(y)\big(\GM^{kl}+e_y(\chi^2_{kl})\big)\right)\,dy.
\end{equation}
So, our limit problem \eqref{L01} becomes using \eqref{514} and \eqref{5171}
\begin{multline*}
	\int_{\O}A^{hom}e(\Gu_0):e(w_0)\,dx+\a\|\Gu_0+{1\over \kappa}\wh W\|^p_{L^2(\O\X Y)}\int_{\O\X Y}(\Gu_0+{1\over \kappa}\wh W)\cdot w_0\,dydx\\
	=\int_{\O}f\cdot w_0\,dx+|Y_1|\int_\O\Mc_{Y_1}(\wh\Theta)\cdot w_0\,dx.
\end{multline*}
Similarly, as in Theorem \ref{ENH01}, the homogenized limit stiffness tensor satisfies the following as a direct consequence of the hypothesis on $A_{ijkl}$ and the inequality \eqref{Ex01} (also see  Corollary 10.13 in \cite{Donato} and Theorem II.1.1 in \cite{Olenik}).
\begin{lemma}\label{HT01}
	We have $A^{hom}_{ijkl}\in\R$ and
	\begin{itemize}
		\item Symmetry: $A^{hom}_{ijkl}=A^{hom}_{jikl}=A^{hom}_{klij}$.
		\item Ellipticity and uniform boundedness: There exists  $0<C_1<C_2$ such that
		\begin{equation}\label{424}
			\begin{aligned}
				\text{for all $\zeta\in \R^{N\X N}_\sy$},&\quad A_{ijkl}^{hom}\zeta_{kl}\zeta_{ij}\geq C_1|\zeta|^2_F,\\
				\text{for all $\zeta\in \R^{N\X N}$},&\quad A_{ijkl}^{hom}\zeta_{kl}\leq C_2|\zeta|_F. 
			\end{aligned}
		\end{equation}
	\end{itemize}
\end{lemma}
Similarly, choosing $w_0=0$ and $W_2=0$ in \eqref{L01}, with localizing, we obtain
\begin{equation}\label{5171}
	\begin{aligned}
		&\text{Find $\wh W\in \GH^1_{per}(Y)^N$ such that for all $W_1\in \GH^1_{per}(Y)^N$}\\
		&\hskip 10mm \int_{Y} A(y)e_y(\wh W): e_y(W_1)\,dy+\alpha\|\Gu_0+{1\over \kappa}\wh W\|^p_{L^2(\O\X Y)}\int_{Y}(\Gu_0+{1\over \kappa}\wh W)\cdot {1\over \kappa}W_1\,dy\\
		&\hskip 100mm={1\over \kappa}\int_{Y}(f+\1_{\O\X Y_1}{\wh \Theta})\cdot W_1\,dy.
	\end{aligned}
\end{equation}
When $\kappa=+\infty$. Choosing $w_0=0$ and $W_2=0$ in \eqref{L01}, with localizing, we obtain 
\begin{equation}\label{5172}
	\begin{aligned}
		&\text{Find $\wh W\in \GH^1_{per}(Y)^N$ such that}\\
		&\hskip 10mm \int_{Y} A(y)e_y(\wh W): e_y(W_1)\,dy=0,\quad\text{for $W_1\in \GH^1_{per}(Y)^N$},
	\end{aligned}
\end{equation}
which has a unique solution using Lax-Milgram lemma and is given by $\wh W=0$. Similarly in \eqref{L01} we replace $w_0=0$ and $W_1=0$ with localizing, we obtain \eqref{514}.
\begin{remark}\label{ReM02}
	Note that the cell problems defined on $Y_2$ for any $\kappa\in(0,+\infty]$ are all the same, and this common problem is given by~\eqref{514}. Moreover, the two-scale limit problem~\eqref{L01} for $\kappa=+\infty$ can be obtained by first passing to the limit $(\varepsilon,\delta)\to(0,0)$ with $\kappa\in(0,+\infty)$ to derive~\eqref{L01}, and then passing to the limit $\kappa\to+\infty$.
\end{remark}
%
Finally, using the two-scale energy \eqref{EXE01} and the cell problems \eqref{514}, \eqref{5171}, \eqref{5172} with the expression \eqref{Hom01}, Theorem \ref{ENH01} and Lemma \ref{HT01}, we have the following
\begin{theorem}[Homogenized system]
	\label{56}
	For $\kappa\in(0,+\infty]$ and let $A^{hom}$ be given by \eqref{Hom01}. 
	Then the limit displacement $(\Gu_0,\wh W)\in \D_0$ is the unique minimizer of the homogenized energy functional
	$\GT^{hom}_\kappa:\D_0\to\R$, and
	\begin{equation}\label{425}
		\Gm_\kappa=\GT^{hom}_\kappa(\Gu_0,\wh W),
	\end{equation}
	where $\GT^{hom}_\kappa$ is given by
	\begin{multline}\label{Thm:HomEnergy}
		\GT^{hom}_\kappa(\Gu_0,\wh W)
		=
		\frac12\int_\O A^{hom}e(\Gu_0):e(\Gu_0)\,dx+{1\over 2}\int_{\O\X Y}A(y) e_y(\wh W):e_y(\wh W)\,dxdy\\
		+\frac{\a}{p+2}\Big\|\Gu_0+\frac1\kappa\wh W\Big\|_{L^2(\O\times Y)}^{p+2}
		-\int_{\O\times Y}\big(f+\1_{\O\X Y_1}\wh\Theta\big)\cdot\Big(\Gu_0+\frac1\kappa\wh W\Big)\,dx\,dy.
	\end{multline}
	Moreover, $(\Gu_0,\wh W)\in\D_0$ is the unique solution of 
		\begin{multline*}
		\int_{\O}A^{hom}e(\Gu_0):e(w_0)\,dx+\int_{\O\X Y}A(y) e_y(\wh W):e_y(W_1)\,dxdy\\
		+\a\|\Gu_0+{1\over \kappa}\wh W\|^p_{L^2(\O\X Y)}\int_{\O\X Y}(\Gu_0+{1\over \kappa}\wh W)\cdot (w_0+{1\over \kappa}W_1)\,dydx\\
		=\int_{\O\X Y}(f+\1_{\O\X Y_1}\wh \theta)\cdot (w_0+{1\over \kappa}W_1)\,dydx,\quad\forall\,(w_0,W_1)\in\D_0.
	\end{multline*}
\end{theorem}

\section{Asymptotic behavior of the $\e,\d$-OCP}\label{S06}
%
	The main aim is to analyze the asymptotic behavior of the $\varepsilon$-$\d$ dependent optimal control problem as $(\e,\d)\to(0,0)$ satisfying \eqref{R1}, for which first we need a proper characterization of the optimal control.
	
	From, Theorem \ref{Th2}, we have there exist an optimal control $\Theta_\e$ such that
	$$\Gj_\ed(\Theta_\e)=\inf_{\theta_\e\in\GV}\Gj_\ed(\theta_\e),\quad \|\Theta_\e\|_{L^2(\O)}\leq C.$$
	The constant is independent of $\e$ and $\d$.
	
	The adjoint system w.r.t the unique solution $\wo u_\ed$ corresponding to $\Theta_\e$ is given by
	\begin{equation}\label{AdjW01}
		\begin{aligned}
			&\text{Find $\wo v_\ed\in \GU$ such that for all $w\in\GU$}\\
			&\hskip 10mm \Ga_\ed(w,\wo v_\ed)+\a\|\wo u_\ed\|^p_{L^2(\O)}(w,\wo v_\ed)+\a p\|\wo u_\ed\|^{p-2}_{L^2(\O)}(\wo u_\ed,w)(\wo u_\ed,\wo v_\ed)=(\wo u_\ed-u_d,w),	
		\end{aligned}
	\end{equation}
	where $(\cdot,\cdot)$ denotes the $L^2(\O)$ inner product. In the energy form the adjoint system is given by: find $\wo v_\ed\in\GU$ such that
	\[
	\Gm^\ast_\ed=\Jc_\ed(\wo v_\ed)=\min_{v\in\GU}\mathcal J_\ed(v),
	\]
	where the quadratic functional $\mathcal J_\ed:\GU\to\R$ by
	\begin{equation}\label{AdjEnergy}
		\mathcal J_\ed(v)
		=
		\frac12\,\Ga_\ed(v,v)
		+\frac{\a}{2}\,\|\wo u_\ed\|_{L^2(\O)}^{p}\,\|v\|_{L^2(\O)}^{2}
		+\frac{\a p}{2}\,\|\wo u_\ed\|_{L^2(\O)}^{p-2}\,(\wo u_\ed,v)^2
		-\;(\wo u_\ed-u_d,\,v),
		\quad v\in\GU.
	\end{equation}
	The Euler--Lagrange condition $\frac{d}{dt}\mathcal J_\ed(\wo v_\ed+t w)\big|_{t=0}=0$
	for all $w\in\GU$ is exactly \eqref{AdjW01}.
	\begin{theorem}
		There exist a unique minimizer of the adjoint energy \eqref{AdjEnergy}. The unique minimizer $\wo v_\ed\in\GU$ satisfy the following:
		\begin{equation}\label{EsA01}
			\d\|e(\wo v_\ed)\|_{L^2(\O^1_\e)}+\|e(\wo v_\ed)\|_{L^2(\O^2_\e)}\leq C\|\wo u_\ed-u_d\|_{L^2(\O)}.
		\end{equation}
		The constant is independent of $\e$ and $\d$.
	\end{theorem}
	\begin{proof}
		Fix $\wo u_\ed\in \GU$ and since $p\ge 2$. 
		Define the bilinear form $B_\ed:\GU\times\GU\to\R$ and the linear functional $F_\ed:\GU\to\R$ by
		\begin{align*}
			B_\ed(w,v)
			&=\Ga_\ed(w,v)
			+\a\|\wo u_\ed\|_{L^2(\O)}^{p}\,(w,v)
			+\a p\,\|\wo u_\ed\|_{L^2(\O)}^{p-2}\,(\wo u_\ed,w)\,(\wo u_\ed,v),\quad
			F_\ed(w)
			=(\wo u_\ed-u_d,w).
		\end{align*}
		Then the adjoint problem \eqref{AdjW01} can be written as
		\begin{equation}\label{AdjW01_LM}
			\text{Find $\wo v_\ed\in\GU$ such that } \quad B_\ed(w,\wo v_\ed)=F_\ed(w)\quad\forall w\in\GU.
		\end{equation}
		By the continuity of $\Ga_\ed$ on $\GU\times\GU$ and the Cauchy--Schwarz inequality in $L^2(\O)$,
		there exists $C>0$ (independent of $w,v$) such that
		\begin{multline*}
			|B_\ed(w,v)|
			\le C\|w\|_{H^1(\O)}\|v\|_{H^1(\O)}\\
			+\a\|\wo u_\ed\|_{L^2(\O)}^{p}\,\|w\|_{L^2(\O)}\|v\|_{L^2(\O)}
			+\a p\,\|\wo u_\ed\|_{L^2(\O)}^{p-2}\,\|\wo u_\ed\|_{L^2(\O)}^2\,\|w\|_{L^2(\O)}\|v\|_{L^2(\O)}.
		\end{multline*}
		Using the continuous embedding $H^1(\O)^N\hookrightarrow L^2(\O)^N$, there exists $C_0>0$
		such that $\|z\|_{L^2(\O)}\le C_0\|z\|_{H^1(\O)}$ for all $z\in\GU$. Hence
		\[
		|B_\ed(w,v)|\le C_1\|w\|_{H^1(\O)}\|v\|_{H^1(\O)}\qquad\forall w,v\in\GU,
		\]
		for some constant $C_1>0$. Similarly,
		\[
		|F_\ed(w)|=|(\wo u_\ed-u_d,w)|\le \|\wo u_\ed-u_d\|_{L^2(\O)}\,\|w\|_{L^2(\O)}
		\le C_0\|\wo u_\ed-u_d\|_{L^2(\O)}\,\|w\|_{H^1(\O)},
		\]
		so $F_\ed$ is continuous on $\GU$.

		Since $p\ge 2$, $\|\wo u_\ed\|_{L^2(\O)}^{p}\ge 0$ and $\|\wo u_\ed\|_{L^2(\O)}^{p-2}\ge 0$. Thus,
		for any $v\in\GU$,
		\[
		B_\ed(v,v)
		=\Ga_\ed(v,v)+\a\|\wo u_\ed\|_{L^2(\O)}^{p}\,\|v\|_{L^2(\O)}^{2}
		+\a p\,\|\wo u_\ed\|_{L^2(\O)}^{p-2}\,(\wo u_\ed,v)^2
		\ge \Ga_\ed(v,v).
		\]
		By the coercivity assumption on $\Ga_\ed$, there exists $c>0$ such that
		\[
		\Ga_\ed(v,v)\ge c\,\|v\|_{H^1(\O)}^2\qquad\forall v\in\GU\implies B_\ed(v,v)\ge c\,\|v\|_{H^1(\O)}^2\qquad\forall v\in\GU.
		\]
		Therefore, $B_\ed$ is coercive on $\GU$.

		By the Lax--Milgram theorem, there exists a unique $\wo v_\ed\in\GU$ such that \eqref{AdjW01_LM} holds, i.e.
		$\wo v_\ed$ is the unique solution of the adjoint variational problem \eqref{AdjW01}.
%
		Since $B_\ed$ is coercive, $\mathcal J_\ed$ is strictly convex and therefore has at most one minimizer.
		Consequently, the unique solution $\wo v_\ed$ given by Lax--Milgram is the unique minimizer of $\mathcal J_\ed$.
		
		Finally, taking $w=\wo v_\ed$ in \eqref{AdjW01}, we get
		$$C\d^2\|e(\wo v_\ed)\|^2_{L^2(\O^1_\e)}+\|e(\wo v_\ed)\|^2_{L^2(\O^2_\e)}\leq C\|\wo u_\ed-u_d\|_{L^2(\O)}\|\wo v_\ed\|_{L^2(\O)}.$$
		Then, with the inequality \eqref{NK01} and proceeding as in Lemma \ref{L3}, we get \eqref{EsA01}.
		This completes the proof.
	\end{proof}

	Below, we give the necessary condition and converse statement of the first order optimality.		
		\begin{theorem}[First-order necessary conditions]
			\label{CC01}
			Let $\wo u_\ed\in\GU$ be the (unique) state associated with an optimal control $\Theta_\e\in\GV$ for \eqref{MainW01}.
			Let $\wo v_\ed\in\GU$ solve the adjoint problem \eqref{AdjW01}, for the fixed $\wo u_\ed$. 
			Then the optimal control satisfies
			\begin{equation}\label{OCF01}
				\Theta_\e=-\frac{1}{\gamma}\,\wo v_\ed \qquad \text{in }\Omega^1_\e.
			\end{equation}
			Conversely, if $(\wo u_\ed,\wo v_\ed,\Theta_\e)\in\GU\times\GU\times\GV$ satisfies, for all $w\in\GU$,
			\begin{equation}\label{OS01}
				\begin{aligned}
					&\Ga_\ed(\wo u_\ed,w)+\alpha\|\wo u_\ed\|_{L^2(\Omega)}^{p}\,(\wo u_\ed,w)=\ell_\e(w),\\
					&\Ga_\ed(w,\wo v_\ed)
					+\alpha\|\wo u_\ed\|_{L^2(\Omega)}^{p}\,(w,\wo v_\ed)
					+\alpha p\,\|\wo u_\ed\|_{L^2(\Omega)}^{p-2}\,(\wo u_\ed,w)\,(\wo u_\ed,\wo v_\ed)
					= \,(\wo u_\ed-u_d,w),\\
					&\Theta_\e=-\frac{1}{\gamma}\,\wo v_\ed,
				\end{aligned}
			\end{equation}
			then $\Theta_\e$ is a stationary point of the OCP.
		\end{theorem}
		\begin{proof}			
			\textbf{Step 1:} Linearization of the state map.
			
			Fix $\theta_\e\in\GV$ and denote $u_\ed=u_\ed(\theta_\e)\in \GU$ the unique state solving \eqref{MainW01}. For any $h\in\GV$, the Fréchet derivative $\dot{u}_\ed=\Sc_\ed'(\theta_\e)h\in \GU$ exists and solves the
			linearized state problem
			\begin{equation}\label{lin-state}
				\Ga_\ed(\dot{u}_\ed,w)+\alpha\|u_\ed\|_{L^2(\O)}^{p}\,(\dot{u}_\ed,w)
				+\alpha p\,\|u_\ed\|_{L^2(\O)}^{p-2}\,(u_\ed,\dot{u}_\ed)(u_\ed,w)
				=(h,w)\quad\forall w\in \GU.
			\end{equation}
			Coercivity of $\Ga_\ed$ (see \eqref{l2}) on $\GU$ using Korn's inequality and uniform ellipticity together with  boundedness of the lower-order terms yield existence and uniqueness by Lax--Milgram.
			
			\textbf{Step 2:} Adjoint and duality identity.
			
			For the fixed state $u_\ed$, let $v_\ed\in \GU$ solve the adjoint problem \eqref{AdjW01} for the fixed state $u_\ed$ and control $\theta_\e\in\GV$.
			Choose $w=\dot{u}_\ed$ in \eqref{AdjW01} and  $w=v_\ed$ in \eqref{lin-state}; since the left-hand sides coincide (due to the fact that $\Ga_\ed$ is symmetric), we obtain
			\[
			(u_\ed-u_d,\dot{u}_\ed)=(h,v_\ed).
			\]
			
			\textbf{Step 3:} Derivative of the reduced cost and optimality condition.
			
			By the chain rule,
			\[
			\Gj_\ed'(\theta_\e)[h]=(u_\ed-u_d,\dot{u}_\ed)+\gamma(\theta_\e,h)
			=(h,v_\ed)+\gamma(\theta_\e,h)
			=(\gamma\theta_\e+v_\ed,h)\quad\forall h\in\GV.
			\]
			Since $\Theta_\e$ is an unconstrained optimal control, the stationary condition $\Gj_\ed'(\Theta_\e)[h]=0$ for all $h\in\GV$ implies
			\[
			(\gamma\Theta_\e+\wo v_\ed,h)=0\quad\forall h\in\GV
			\quad\Rightarrow\quad
			\gamma\Theta_\e+\wo v_\ed=0\ \text{ in $\Omega^1_\e$}
			\]
			i.e.~\eqref{OCF01} holds.
			
			{\bf Step 4:} Stationary condition.
			
			Conversely, if $(\wo u_\ed,\wo v_\ed,\Theta_\e)$ satisfies the state equation, the adjoint equation, and
			$\Theta_\e=-\wo v_\ed/\gamma$, then for every $h\in\GV$ the computation above gives
			\[
			\Gj_\ed'(\Theta_\e)[h]=(\gamma\Theta_\e+\wo v_\ed,h)=0,
			\]
			so $\Theta_\e$ is a stationary point of the OCP. This completes the proof.
		\end{proof}
		\begin{remark}[On sufficiency]
			In general with the nonlinear term $\alpha\|u\|_{L^2(\O)}^p u$, the system \eqref{OS01} gives
			first-order necessary conditions. Since, for $\alpha>0$ the control-to-state map $\Sc_\ed:\GV\to\GU$ is nonlinear (indeed, non-affine). 
			A quick scaling argument shows this: for simplicity we take $f\equiv 0$. If $u=\Sc_\ed(\theta)$ and $t>0$, then 
			$u_t:=\Sc_\ed(t\theta)$ solves
			\[
			\Ga_\ed(u_t,w)+\alpha\|u_t\|_{L^2(\Omega)}^{p}(u_t,w)=(t\theta,w)\quad\forall w\in\GU.
			\]
			If $\Sc_\ed$ were linear, we would have $u_t=tu$, hence
			\[
			\Ga_\ed(tu,w)+\alpha\|tu\|_{L^2(\O)}^{p}(tu,w)
			=t\,\Ga_\ed(u,w)+t^{p+1}\alpha\|u\|_{L^2(\O)}^{p}(u,w),
			\]
			which equals $t\big(\Ga_\ed(u,w)+\alpha\|u\|_{L^2(\O)}^{p}(u,w)\big)$ only when $\alpha=0$ or $u\equiv 0$. 
			Thus $\Sc_\ed$ is not linear for $\alpha>0$.
			As a consequence, the reduced functional 
			\[
			\Gj_\ed(\theta)=\tfrac12\|\Sc_\ed(\theta)-u_d\|_{L^2(\Omega)}^2+\tfrac{\gamma}{2}\|\theta\|_{L^2(\Omega)}^2
			\]
			is, in general, non-convex for $\alpha>0$. Indeed, differentiating twice in a direction $h\in\GV$ yields
			\[
			\Gj_\ed''(\theta)[h,h]
			=\|\Sc_\ed'(\theta)h\|_{L^2(\Omega)}^2
			+\big(\Sc_\ed(\theta)-u_d,\;\Sc_\ed''(\theta)[h,h]\big)_{L^2(\Omega)}
			+\gamma\|h\|_{L^2(\Omega)}^2,
			\]
			and the middle term has no fixed sign in general. In the linear case $\alpha=0$ we have $\Sc_\ed''\equiv 0$, so
			\[
			\Gj_\ed''(\theta)[h,h]=\|\Sc_\ed'(\theta)h\|_{L^2(\O)}^2+\gamma\|h\|_{L^2(\O)}^2\ge \gamma\|h\|_{L^2(\O)}^2,
			\]
			hence $\Gj_\ed$ is (strictly) convex and the optimal control is unique. For $\alpha>0$, nonconvexity may lead to
			multiple stationary points and non-uniqueness of global minimizers.
		\end{remark}
	From the above theorem, it is sufficient to analyze the asymptotic behavior of the adjoint state $\wo v_\ed$ to get the strong convergence of the optimal control. So, below, we first present the asymptotic behavior of the adjoint state,  followed by the derivation of two-scale adjoint problem, finally we end with the limit optimal control problem. 
	
	Let $\wo v_\ed$ be the unique adjoint state of the problem \eqref{AdjW01}. Then, using Lemma \ref{L2}, there exist $\wo\Gv_\ed\in \GU$ satisfying  \eqref{46}$_{1,3}$ and $\wo\Gz_\ed\in \GU$ satsifying \eqref{46}$_{2,4,5}$ such that
	$$\wo v_\ed=\wo\Gv_\ed+\wo\Gz_\ed,\quad \text{in $\O$}.$$
	 Then, using the estimates \eqref{46}--\eqref{E01}, \eqref{UBOC}--\eqref{UBOCC} and \eqref{EsA01}, we obtain
	\begin{equation}\label{61}
		\|\wo \Gv_\ed\|_{H^1(\O)}+{\d\over \e}\|\wo\Gz_\ed\|_{L^2(\O^1_\e)}+\d\|\nabla\wo\Gz_\ed\|_{L^2(\O^1_\e)}\leq C\|\wo u_\ed-u_d\|_{L^2(\O)}\leq C.
	\end{equation}
	The constant is independent of $\e$ and $\d$.
	\begin{lemma}[Two-scale adjoint system]\label{L08}
		Let $(\bar{u}_\ed,\Theta_\e)\in\GU\X \GV$ be the optimal pair to the OCP \eqref{OCPM01} and $\wo v_\ed\in\GU$ be the adjoint state. Then, there exist $(\Gv_0,\wh Z,\wh V)\in\D$ such that
		\begin{equation}\label{Con011}
			\begin{aligned}
				\Te(\wo\Gv_\ed)&\to \Gv_0, &&\text{strongly in $L^2(\O\X Y)^N$},\\
				\Te(\nabla \wo\Gv_\ed) &\rightharpoonup \nabla \Gv_0+\nabla_y\wh V, &&\text{weakly in $L^2(\O\X Y)^{N\X N}$},\\
				 \Te(\wo\Gz_\ed) &\rightharpoonup{1\over \kappa}\wh Z,&& \text{weakly in $L^2(\O\X Y)^N$},\\
				 \d\Te(\nabla \wo\Gz_\ed) &\rightharpoonup \nabla_y\wh Z,&&\text{weakly in $L^2(\O\X Y)^{N\X N}$}.
			\end{aligned}
		\end{equation}
		Moreover, we have
		\begin{equation}\label{ECA01}
			\lim_{(\e,\d)\to(0,0)}\Jc_\ed(\wo v_\ed)=\Jc_\kappa(\Gv_0,\wh Z,\wh V)=\min_{(w_0, W_1,W_2)\in\D}\Jc(w_0,W_1,W_2)=\Gm^\ast_\kappa,
		\end{equation}
		where the two-scale adjoint energy is given by
		\begin{multline*}
			\Jc(\Gv_0,\wh Z,\wh V)={1\over 2}\left[\Ga_1(\wh Z)+\Ga_2(\Gv_0,\wh V)\right]+{\a\over 2}\left[\|\Gu_\kappa\|^p_{L^2(\O\X Y)}\|\Gv_\kappa\|^2_{L^2(\O\X Y)}+p\|\Gu_\kappa\|^{p-2}_{L^2(\O\X Y)}(\Gu_\kappa,\Gv_\kappa)^2\right]\\
			-(\Gu_\kappa-u_d,\Gv_\kappa),
		\end{multline*}
		with
		$$\Gu_\kappa=\Gu_0+{1\over \kappa}\wh W,\quad \Gv_\kappa=\Gv_0+{1\over\kappa}\wh Z,\quad\text{in $\O\X Y$}.$$
		Furthermore, we have the strong convergence for the microscopic optimal control
		\begin{equation}\label{OCC03}
				\Te(\Theta_\e) \to -{1\over \gamma}\left(\Gv_0+{1\over \kappa}\wh Z\right)=\wh \Theta,\quad\text{strongly in $L^2(\O\X Y_1)^N$}.
		\end{equation}
	\end{lemma}
	\begin{proof}
		The convergences \eqref{Con011} are obtained using the estimate \eqref{61} as in Lemma \ref{L04}.
		
		Proceeding as in Theorem \ref{Th3}, we get the existence of unique minimizer of $\Jc_\kappa$ in $\D$, i.e.
		$$\min_{(w_0, W_1,W_2)\in\D}\Jc(w_0,W_1,W_2)=\Gm^\ast_\kappa\in(-\infty,0],$$
		then proceeding as in Theorem \ref{ENH01}, we obtain
		$$\lim_{(\e,\d)\to(0,0)}\Jc_\ed(\wo v_\ed)=\Jc_\kappa(\Gv_0,\wh Z,\wh V)=\Gm^\ast_\kappa,$$
		using the strong convergence \eqref{SCM}, weak convergences \eqref{Con011} and the recovery sequence $\{w_\ed\}_{\e,\d}\subset \GU$ from Step 2 of Theorem \ref{ENH01}. This gives \eqref{ECA01}.
		
		Finally, proceeding as in Corollary \ref{Cor1} and the total energy convergence \eqref{ECA01} combined with the convergences \eqref{SCM}, \eqref{Con011} give
		$$
		\lim_{(\e,\d)\to(0,0)}\left[\|\Te(\wo u_\ed)\|^p_{L^2(\O\X Y)}\|\Te(\wo v_\ed)\|^2_{L^2(\O\X Y)}\right]=\|\Gu_\kappa\|^p_{L^2(\O\X Y)}\|\Gv_\kappa\|^2_{L^2(\O\X Y)},$$
		which combined with strong convergence \eqref{SCM} and weak convergences \eqref{Con011}$_{1,3}$ imply the strong convergence
		\begin{equation}\label{SCM+}
			\Te\big(\wo v_\ed\big) \to \Gv_\kappa=\Gv_0+{1\over \kappa}\wh Z,\quad\text{strongly in $L^2(\O\X Y)^{N}$}.
		\end{equation}
		Now, we give the unfolded limit of the optimal control. Using Theorem \ref{CC01}, we have 
		$$\Theta_\e=\begin{aligned}
			&-{1\over \gamma}\wo v_\ed\quad &&\text{a.e. in $\O^1_\e$}.
		\end{aligned}
		$$
		So, using the convergences \eqref{OCC01} and \eqref{SCM+}, we get the strong convergence \eqref{OCC03} for $\kappa\in(0,+\infty]$.
		 This completes the proof.
%
	\end{proof}
	The unique minimizer $(\Gv_0,\wh Z,\wh V)\in\D$ of $\Jc_\kappa$ satisfy the following two-scale system (in variational form)
	\begin{multline*}
		\int_{\O\X Y_2} A(y)\big(e(\Gv_0)+e_y(\wh V)\big):\big(e(w_0)+e_y(W_2)\big)\,dx dy\\
		+\int_{\O\X Y}A(y) e_y(\wh Z):e_y(W_1)\,dxdy+\alpha\|\Gu_\kappa\|^p_{L^2(\O\X Y)}\int_{L^2(\O\X Y)}\Gv_\kappa\cdot\Gw_\kappa\,dydx\\
		+\a p\|\Gu_\kappa\|^{p-2}_{L^2(\O\X Y)}\int_{\O\X Y}\Gu_\kappa\cdot\Gw_\kappa\,dydx\int_{\O\X Y}\Gu_\kappa\cdot\Gv_\kappa\,dydx\\
		=\int_{\O\X Y}(\Gu_\kappa-u_d)\cdot\Gw_\kappa\,dydx,
	\end{multline*}
	for all $(w_0,W_1,W_2)\in\D$ with $\Gw_\kappa=w_0+{1\over \kappa}W_1$.
	
	Then, proceeding as in Subsection \ref{SS44}, we can express $\wh V$ in terms of $\Gv_0$ and correctors $\chi^2_{ij}$ to get the two-scale limit adjoint system for $(\Gv_0,\wh Z)\in\D_0$
	\begin{equation}\label{LAS01}
		\begin{aligned}
		\int_{\O} A^{hom}e(\Gv_0)&:e(w_0)\,dx
		+\int_{\O\X Y}A(y) e_y(\wh Z):e_y(W_1)\,dxdy\\
		&+\a p\|\Gu_\kappa\|^{p-2}_{L^2(\O\X Y)}\int_{\O\X Y}\Gu_\kappa\cdot\Gw_\kappa\,dydx\int_{\O\X Y}\Gu_\kappa\cdot\Gv_\kappa\,dydx\\
		&\hskip 10mm +\alpha\|\Gu_\kappa\|^p_{L^2(\O\X Y)}\int_{L^2(\O\X Y)}\Gv_\kappa\cdot\Gw_\kappa\,dydx=\int_{\O\X Y}(\Gu_\kappa-u_d)\cdot\Gw_\kappa\,dydx.
		\end{aligned}
	\end{equation}
	Now, we are in position to give the limit OCP.
	\subsection{Limit optimal control problem}
	 The two-scale limit OCP is given by
	\begin{equation}\label{LOCP}
		\Gi_\kappa=\inf_{\wh \theta\in L^2(\O\X Y_1)^N}\left\{\Gj_\kappa(\wh\theta)=\GJ_\kappa(v_0,V_1,\wh\theta)\,|\,\text{subjected to $(v_0,V_1)\in\D_0$ satisfying \eqref{TS01}}\right\},
	\end{equation}
	where
	$$\GJ_\kappa(v_0,V_1,\wh\theta)= \left\|v_0+{1\over \kappa}V_1-u_d\right\|_{L^2(\O\X Y)}^2+\frac{\gamma}{2}\left\|\wh\theta\right\|^2_{L^2(\O\X Y_1)} ,$$
	and $(v_0,V_1)\in\D_0$ is the unique solution of
	\begin{multline}\label{TS01}
		\int_{\O}A^{hom}e(v_0):e(w_0)\,dx+\int_{\O\X Y}A(y) e_y(V_1):e_y(W_1)\,dxdy\\
		+\a\|v_0+{1\over \kappa}V_1\|^p_{L^2(\O\X Y)}\int_{\O\X Y}(v_0+{1\over \kappa}V_1)\cdot (w_0+{1\over \kappa}W_1)\,dydx\\
		=\int_{\O\X Y}(f+\1_{\O\X Y_1}\wh \theta)\cdot (w_0+{1\over \kappa}W_1)\,dydx,\quad\forall\,(w_0,W_1)\in\D_0,
	\end{multline}
	 corresponding to $\wh\theta\in L^2(\O\X Y_1)^N$.

	\begin{theorem}
		For each $\kappa\in(0,+\infty]$, the OCP \eqref{LOCP} admits a solution. The limit control to state map $\Sc_\kappa:L^2(\O\X Y_1)^N\to \D_0$ is Lipschitz continuous.
	\end{theorem}
	\begin{proof}
		Observe that the unique solution $(v_0,V_1)\in\D_0$ of \eqref{TS01} corresponding to $\wh \theta$ satisfy
		$$\|v_0\|_{H^1(\O)}+\|V_1\|_{L^2(\O;H^1(Y))}\leq C(\|f\|_{L^2(\O)}+\|\wh \theta\|_{L^2(\O\X Y_1)}),$$
		due to the inequality \eqref{LES01}. 
		
		As
		$$\Gi_\kappa=\inf_{\wh \theta\in L^2(\O\X Y_1)^N}\Gj_\kappa(\wh \theta),$$
		we have $\Gi_\kappa\in[0,+\infty)$ since $\Gi_\kappa\leq \Gj_\kappa(0)<+\infty$. So, proceeding as in the proof of Theorem \ref{Th2}, we get the existence of optimal control for the limit OCP \eqref{LOCP}.
		
		Finally, using the properties of the homogenized stiffness tensor $A^{hom}$ from Lemma \ref{HT01} and proceeding as in Theorem \ref{Th03}, we get the limit control to state mapping is Lipschitz continuous.
	\end{proof}	
	
	Observe that the adjoint system w.r.t the unique solution $(\Gu_0,\wh W)\in\D_0$ of \eqref{TS01} corresponding to $\wh\Theta$ is given by \eqref{LAS01}. The adjoint system \eqref{LAS01} has unique solution $(\Gv_0,\wh Z)\in\D_0$. Now, we prove that $\wh \Theta\in L^2(\O\X Y_1)^N$ is an optimal control of the limit OCP \eqref{LOCP}.
	\begin{theorem}
		The limit of the microscopic optimal control $\wh\Theta\in L^2(\O\X Y_1)^N$ is an optimal control of the limit OCP \eqref{LOCP}, i.e.
		$$\Gj_\kappa(\wh\Theta)=\Gi_\kappa.$$
	\end{theorem} 
	\begin{proof}
		To prove the main result of this theorem, we use a form of $\Gamma$-convergence, to derive the following energy convergence
		\begin{equation}\label{LM}
			\Gj_\kappa(\wh \Theta)=\lim_{(\e,\d)\to(0,0)}\Gj_\ed(\Theta_\e)=\lim_{(\e,\d)\to(0,0)}\Gi_\ed=\Gi_\kappa.
		\end{equation}
		First, due to the strong convergences \eqref{SCM} and \eqref{SCM+}, we get
		\begin{equation}\label{LI}
			\Gi_\kappa\leq \Gj_\kappa(\wh\Theta)=\liminf_{(\e,\d)\to(0,0)}\Gj_\ed(\Theta_\e)=\liminf_{(\e,\d)\to(0,0)}\Gi_\ed.
		\end{equation}
		Now, we prove
		\begin{equation}\label{LS}
			\lim_{(\e,\d)\to(0,0)}\Gi_\ed\leq \Gj_\kappa(\wh \theta),\quad \forall\,\wh \theta\in L^2(\O\X Y_1)^N.
		\end{equation}
		Let $\wh \theta\in L^2(\O\X Y_1)^N$, then there exist a unique state $(v_0,V_1)\in\D_0$, which is the unique solution of \eqref{TS01}, corresponding to $\wh\theta$. 
		
		Let us set
		\begin{equation*}
			\begin{aligned}
				\theta_\e(x)&=\wh\theta\left(x,\left\{x\over \e\right\}\right) 
			\end{aligned},\quad x\in \O^1_\e.
		\end{equation*}
		Observe that $\theta_\e\in\GV$ and using the properties of the unfolding operator, we have
		\begin{equation*}
			\begin{aligned}
				\Te(\theta_\ed)&\to \wh\theta,\quad &&\text{strongly in $L^2(\O\X Y_1)^N$}.
			\end{aligned}
		\end{equation*} 
		Observe that, there exist a unique solution $v_\ed\in\GU$ of \eqref{MainW01} corresponding to $\theta_\ed\in\GV$. Proceeding as Theorem \ref{ENH01} and Corollary \ref{Cor1}, we get: there exist
		$(v_0',V_1')\in\D_0$ such that 
		\begin{equation}\label{strong-to-some-limit}
			T_\e(v_{\e\d})\to v_0'+\frac1\kappa V_1'
			\quad\text{strongly in }L^2(\Omega\times Y)^N.
		\end{equation}
				Using $w_\ed$ from Step 2 of Theorem \ref{ENH01} as test function in the state equation \eqref{MainW01} satisfied by $v_{\e\d}$, corresponding to $\theta_\e$, give
		\begin{equation}\label{weak-eps}
			\Ga_{\e\d}(v_{\e\d},w_\ed)
			+\alpha\|v_{\e\d}\|_{L^2(\Omega)}^{p}\int_\Omega v_{\e\d}\cdot w_\ed\,dx
			=\int_\Omega f\cdot w_\ed\,dx+\int_{\Omega_\e^1}\theta_\e\cdot w_\ed\,dx.
		\end{equation}
				We now pass to the limit in each term, using the unfolding operator and the strong convergences \eqref{Test01} and \eqref{strong-to-some-limit}, to get
			\begin{multline*}	\int_{\O}A^{hom}e(v'_0):e(w_0)\,dx+\int_{\O\X Y}A(y) e_y(V'_1):e_y(W_1)\,dxdy\\
			+\a\|v'_0+{1\over \kappa}V'_1\|^p_{L^2(\O\X Y)}\int_{\O\X Y}(v'_0+{1\over \kappa}V'_1)\cdot (w_0+{1\over \kappa}W_1)\,dydx\\
			=\int_{\O\X Y}(f+\1_{\O\X Y_1}\wh \theta)\cdot (w_0+{1\over \kappa}W_1)\,dydx.
		\end{multline*}
		To obtained the above, we have proceeded as in Section \eqref{SS44}, specifically Theorem \ref{56}. So, we got the exact limit variational formulation \eqref{TS01}
		with the pair $(v_0',V_1')$ and the control $\widehat\theta$. Therefore $(v_0',V_1')$ is a solution of \eqref{TS01}
		associated with $\widehat\theta$. By uniqueness of the limit state (for the given $\widehat\theta$),
		we conclude
		\[
		(v_0',V_1')=(v_0,V_1).
		\]
		The above convergence \eqref{strong-to-some-limit} together with the fact that $\Gi_\ed\leq \Gj_\ed(\theta_\e)=\GJ_\e(v_\ed,\theta_\e)$ gives
		$$\limsup_{(\e,\d)\to(0,0)}\Gi_\ed\leq \limsup_{(\e,\d)\to(0,0)}\Gj_\e(\theta_\e)=\limsup_{(\e,\d)\to(0,0)}\GJ_\e(v_\ed,\theta_\e)=\GJ_\kappa(v_0,V_1,\wh\theta)=\Gj_\kappa(\wh \theta).$$
		Since, $\wh \theta\in L^2(\O\X Y_1)^N$ is arbitary, we get \eqref{LS}. Now, taking an optimal control of the OCP \eqref{LOCP} in place of $\wh \theta$ in \eqref{LS}, gives (together with \eqref{LI})
		$$\Gi_\kappa\leq \Gj_\kappa(\wh\Theta)=\liminf_{(\e,\d)\to(0,0)}\Gi_\ed=\limsup_{(\e,\d)\to(0,0)}\Gi_\ed\leq \Gi_\kappa.$$
		So, we obtain \eqref{LM}. This completes the proof.
	\end{proof}

	Finally, we end this section with the first order optimality condition for the limit OCP \eqref{LOCP}.
	\begin{theorem}
		Let $(\Gu_0,\wh W)\in\D_0$ be the unique state associated with an optimal control $\wh \Theta\in L^2(\O\X Y_1)^N$ for \eqref{TS01}. Let $(\Gv_0,\wh Z)\in\D_0$ solve the adjoint system \eqref{LAS01}, for fixed $(\Gu_0,\wh W)$. Then, the optimal control satisfies
		\begin{equation}
			\wh \Theta=-{1\over \gamma}(\Gv_0+{1\over \kappa}\wh Z),\quad \text{in $\O\X Y_1$}.
		\end{equation}
	\end{theorem}
	\begin{proof}
		The result is direct consequence of Lemma \ref{L08}.
	\end{proof}
		The limit optimal control problem for the microscopic OCP \eqref{OCPM01} is given by \eqref{LOCP}. For $\kappa=+\infty$, we have a complete scale separated state equation, adjoint and limit OCP. 

%
%


				%
				
				
				\section*{Acknowledgment}
				The authors would like to thank Prof.~Dr.~Georges Griso and Dr.~Julia Orlik for their suggestions and corrections. Abu Sufian acknowledges support from the Agencia Nacional de Investigaci\'on y Desarrollo, Chile, under the FONDECYT postdoctoral fellowship,  project no ($N^\circ ~3240018$).
				

				
				{\small\bibliographystyle{plain}
				\bibliography{reference}} 
				
				\begin{appendices}
					\section{Appendix}
					In this section, we prove some coercivity results for $N\geq 2$.
					\begin{theorem}
						For $(\Gu_0,\wh U)\in \GU\X L^2(\O; H^1_{per,0}(Y_2))^N$, we have that there exists a $C_0>0$ such that
						\begin{equation}\label{Ex01}
							\|\Gu_0\|^2_{H^1(\O)}+\|\wh U\|_{L^2(\O;H^1(Y_2))}\leq C_0\int_{\O\X Y_2}\big(e(\Gu_0)+e_y(\wh U)\big):\big(e(\Gu_0)+e_y(\wh U)\big)\,dxdy.
						\end{equation}
					\end{theorem}
					The above inequality is a direct consequence of Korn's inequality and the fact that $\int_{Y_2}\nabla_y\wh U\,dy=0$ since $\wh U\in H^1_{per,0}(Y_2)^N$. Using the above inequality we prove using Lax-Milgram lemma that the unfolded limit problems have a unique solution.
				\end{appendices}
\end{document}